\documentclass[10pt]{article} 
\usepackage[utf8]{inputenc}
\usepackage[top=30mm, left=30mm, right=30mm, bottom=30mm]{geometry}
\usepackage{amsmath,amssymb,amsthm}
\usepackage{dsfont}
\usepackage{color}
\usepackage{lineno,hyperref}
\usepackage{yfonts}
\usepackage{appendix}

\modulolinenumbers[5]

\newtheorem{theorem}{Theorem}[section]
\newtheorem{lem}[theorem]{Lemma}
\newtheorem{exa}[theorem]{Example}
\newtheorem{kor}[theorem]{Corollary}
\newtheorem{prop}[theorem]{Proposition}

\theoremstyle{definition}
\newtheorem{dfn}[theorem]{Definition}
\newtheorem{rem}[theorem]{Remark}

\definecolor{orange}{rgb}{1.0, 0.55, 0.0}

\DeclareMathOperator*{\supp}{supp}
\DeclareMathOperator*{\divv}{div}

\DeclareMathOperator*{\Id}{Id}
\DeclareMathOperator*{\loc}{loc}

\DeclareMathOperator*{\comp}{comp}
\DeclareMathOperator*{\dist}{dist}

\begin{document}
	\title{Flow selections for (nonlinear) Fokker--Planck--Kolmogorov equations}

	\author{Marco Rehmeier\footnote{Faculty of Mathematics, Bielefeld University, 33615 Bielefeld, Germany. E-Mail: mrehmeier@math.uni-bielefeld.de }}
	
	\date{}
	\maketitle
	\begin{abstract}
We provide a method to select flows of solutions to the Cauchy problem for linear and nonlinear Fokker--Planck--Kolmogorov equations (FPK equations) for measures on Euclidean space. In the linear case, our method improves similar results of a previous work of the author. Our consideration of flow selections for nonlinear equations, including the particularly interesting case of Nemytskii-type coefficients, seems to be new. We also characterize the (restricted) well-posedness of FPK equations by the uniqueness of such (restricted) flows. Moreover, we show that under suitable assumptions in the linear case such flows are Markovian, i.e. they fulfill the Chapman-Kolmogorov equations.
\end{abstract}

	
\section{Introduction}
We study linear and nonlinear Fokker--Planck--Kolmogorov equations (FPK equations), which are second-order parabolic equations for measures. For coefficients $a_{ij}$ and $b_i$, $1\leq i,j \leq d$, depending on $(t,\zeta,x)$, where $(t,x) \in \mathbb{R}_+\times \mathbb{R}^d$ and $\zeta \in \mathcal{SP}$ is a subprobability measure on $\mathbb{R}^d$, we consider the second-order operator (using Einstein summation convention)
\begin{equation}\label{Def_L_Intro}
	\mathcal{L}_{t,\zeta}\varphi(x) := a_{ij}(t,\zeta,x)\partial_{ij}\varphi(x)+b_{i}(t,\zeta,x)\partial_i\varphi(x), \quad \varphi \in C^2(\mathbb{R}^d).
\end{equation}
To this operator, we associate the following FPK equation for curves $t \mapsto \mu_t$ in $\mathcal{SP}$ in Schwartz distributional sense:
\begin{equation}\label{NLFPKE_Intro}
	\partial_t\mu_t = \mathcal{L}^*_{t,\mu_t}\mu_t,
\end{equation}
where $\mathcal{L}^*$ denotes the formal dual of $\mathcal{L}$. For $(s,\nu) \in \mathbb{R}_+\times \mathcal{SP}$, we consider \eqref{NLFPKE_Intro} with vague Cauchy-type initial condition
\begin{equation}\label{Initial-cond_Intro}
	\mu_s = \nu.
\end{equation}
Often, $a = (a_{ij})_{i,j \leq d}$ and $b= (b_i)_{i \leq d}$ are interpreted as diffusion and drift coefficients, respectively.
The nonlinearity of \eqref{NLFPKE_Intro} arises from the dependence of the coefficients on the solution $\mu_t$. An important example of such nonlinear dependence is the challenging case of \textit{Nemytskii-type} coefficients $a(t,\zeta,x) = \bar{a}(t,(d\zeta / dx)(x),x)$, $b(t,\zeta,x) = \bar{b}(t,(d\zeta / dx)(x),x)$ where $d\zeta / dx$ denotes the density of a measure $\zeta \ll dx$. In the case where $a$ and $b$ are independent of $\zeta$, \eqref{NLFPKE_Intro} is a linear equation for measures. 

Consideration of FPK equations dates back at least to works by Fokker \cite{Fokker14}, Planck \cite{Planck17} and Kolmogorov \cite{Kolmogoroff1,Kolmogoroff2}. In \cite{Kolmogoroff0}, Kolmogorov posed the question whether \eqref{NLFPKE_Intro}-\eqref{Initial-cond_Intro} has a solution $\mu^{s,x} = (\mu^{s,x}_t)_{t \geq 0}$ for each initial condition $(s,\delta_x)$ ($\delta_x$ denotes the Dirac measure in $x \in \mathbb{R}^d)$ such that the \textit{Chapman-Kolmogorov equations}
\begin{equation}\label{CK-eq}
	\mu^{s,x}_t = \int_{\mathbb{R}^d}\mu^{r,y}_td\mu^{s,x}_r(dy),\quad \forall 0\leq s \leq r \leq t,\, x \in \mathbb{R}^d
\end{equation}
hold. We come back to this question in Section 5.
Emerging from these pioneering works, FPK equations, at first primarily in the linear case, have become an intensively studied field in statistical physics, quantum mechanics and stochastic analysis. For an extensive account of (linear) FPK equations, we refer to the monograph \cite{FPKE-book15}.

Recently, the challenging case of nonlinear FPK equations has sparked particular interest from several directions, among them modeling of porous media, neurophysics, population dynamics and computational science \cite{Frank_NLFPKE-book}. If the coefficients are of Nemytskii-type and one identifies a solution $\mu_t$ with its density $\rho_t$, \eqref{NLFPKE_Intro} may be rewritten as the nonlinear PDE in weak sense
\begin{equation}\label{PDE_Intro}
	\partial_t \rho_t(x) = \partial_{ij}(a_{ij}(t,\rho_t(x),x)\rho_t(x))-\partial_i(b_i(t,\rho_t(x),x)\rho_t(x)).
\end{equation}
Choosing $a_{ij} = \delta_{ij}{a}$ for a suitable coefficient $a: \mathbb{R}_+\times\mathbb{R}\times \mathbb{R}^d \to \mathbb{R}$, \eqref{PDE_Intro} becomes a perturbed general Porous Media Equation, and interpreting \eqref{PDE_Intro} as a FPK equation provides an analytic tool to study such equations with completely degenerate measures as initial condition, see the early works \cite{BC79,Pierre82} for the case $b \equiv 0$, and \cite{BR18,NLFPK-DDSDE2,BR18_2,NLFPK-DDSDE5} for recent more general results.

The fruitful connection of (nonlinear) FPK equations to stochastic analysis stems from the equivalence of weak solutions to the (distribution-dependent) Itô stochastic differential equation (SDE) with coefficients $b$ and $\sigma$ and probability measure-valued solutions to \eqref{NLFPKE_Intro} with coefficients $b$ and $a = 1/2\sigma\sigma^T$, see \cite{Ambrosio08,Figalli09,Trevisan16} for linear and \cite{BR18,BR18_2} for nonlinear equations. As a matter of fact, if \eqref{NLFPKE_Intro}-\eqref{Initial-cond_Intro} is linear and well-posed, the solutions to the associated SDE constitute a Markov process. If solutions are not unique, it is an important question whether one can still select one solution for each initial condition such that the selected family forms a Markov process. As far as we know, the first positive answer in this direction is due to Krylov \cite{Krylov_1973}. In the spirit of this pioneering work, Stroock and Varadhan \cite[Ch.12]{StroockVaradh2007} proved existence of a Markov selection for ill-posed martingale problems for $a$ and $b$ bounded, continuous and time-independent. Recall that weak solutions to the SDE with coefficients $a$ and $\sigma$ are equivalent to solutions to the martingale problem (which are probability measures on path space) with coefficients $b$ and $a= 1/2\sigma \sigma^T$. Their selection method is based on the weak compactness of the set of solutions to the martingale problem for fixed initial condition. It is readily seen that the one-dimensional time marginals of such Markovian selections on paths pace fulfill the Chapman-Kolmogorov equations \ref{CK-eq}.

Despite the aforementioned equivalence of linear FPK equations and martingale problems, to the best of our knowledge, our previous work \cite{Rehmeier_Flow-JEE} is the first to consider selections on the level of the FPK equation. It turned out that at least for continuous and bounded coefficients for linear FPK equations, it is possible to select a family of solutions $\mu^{s,\nu}$, $(s,\nu) \in [0,T)\times \mathcal{SP}$, to \eqref{NLFPKE_Intro}-\eqref{Initial-cond_Intro}, which fulfills
\begin{equation}\label{Flow_intro}
	\mu^{s,\nu}_t = \mu^{r,\mu^{s,\nu}_r}_t \quad\forall 0\leq s < r< t < T.
\end{equation}
We call such a family a \textit{flow}.
This selection is based on the compactness of solutions to the FPK equation, which we obtained from the previously mentioned compactness for the associated martingale problem. This approach relied on the linearity of \eqref{NLFPKE_Intro} and seems not to be robust when passing to less restrictive assumptions on $a$ and $b$.

In the present paper, completely leaving the martingale problem out of the picture, we present a more general, purely analytic selection method for solutions to FPK equations with the flow property: first, we recover our results from \cite{Rehmeier_Flow-JEE} and extend them to equations with more general coefficients. Secondly, our present approach also covers nonlinear FPK equations, partially including the challenging case of Nemytskii-type coefficients. Thirdly, we refine the selection method such that it allows to select solutions from suitable subsets of solutions and and to restrict to subclasses of initial conditions.
Finally, we realized that at least under the assumptions of \cite{Rehmeier_Flow-JEE}, properties \eqref{Flow_intro} and \eqref{CK-eq} coincide. Our selection method is still in the spirit of \cite{Krylov_1973} and \cite{StroockVaradh2007}.

More precisely, our main results are Theorems \ref{Thm-1} and \ref{Thm-2}. The latter is a consequence of the proof of the former and characterizes well-posedness of \eqref{NLFPKE_Intro}-\eqref{Initial-cond_Intro} in terms of uniqueness of the selected flow. We obtained such a result already in \cite{Rehmeier_Flow-JEE}, and extend it here to our present more general situation. We formulate these results in a very general topological manner, aiming for a unified formulation, which applies in various situations of interest. Our main applications are Propositions \ref{Prop_entire-flow-SP}, \ref{Prop_entire_prob}, \ref{Lyap flow prop} in the linear case, and Propositions \ref{1_prop_NL_entire-SP}, \ref{1_prop_Application-CG19} and Corollary \ref{Cor_Nemytskii-case} in the nonlinear case. The latter result treats the case of Nemytskii-type drift coefficients $b$ and particularly applies under the assumptions of several recent existence results for nonlinear FPK equations, such as \cite{BR18_2,BR21_NLFP-time-dep,NLFPK-DDSDE5}. We make precise comparisons to these works at the end of Section 4.

Finally, we would like to mention again that the results of this paper can be considered from at least two perspectives: with regard to the theory of PDEs, we select semigroups of distributional solutions to equations of type \eqref{PDE_Intro}.  On the other hand (and more generally), regarding to the connection to stochastic equations, we prove that for \eqref{NLFPKE_Intro}-\eqref{Initial-cond_Intro}, one can select a solution flow, which seems to be a more general concept than the Chapman-Kolmogorov equations \eqref{CK-eq}, which are closely related to Markovian selections. Under certain assumptions on the coefficients of the equation, we show that these notions coincide.

The remainder of the paper is organized as follows. After introducing notation and the main definitions in Section 2, in Section 3 we formulate and prove both main results, Theorems \ref{Thm-1} and \ref{Thm-2}. Section 4 contains linear and nonlinear applications, i.e. here we give concrete sets of conditions for the coefficients under which our results apply. Precisely, Subsection \ref{Subsection_lin-appl} and \ref{Subsection_nonlin-appl} consist of applications to linear and nonlinear equations, respectively. In Section 5, we return to linear equations and prove that for continuous and bounded coefficients, properties \eqref{Flow_intro} and \eqref{CK-eq} coincide. Finally, Appendix A and B contain auxiliary lemmas and proofs, and basics on measurable selections, respectively.
\paragraph{Funding}
Financial support by the CRC 1283 (Bielefeld University) of the German Research Foundation is gratefully
acknowledged.

\section{Notation and definitions}
\subsection{Notation}
For topologies $\tau_1$ and $\tau_2$ on a set $\mathcal{H}$, we write $\tau_1 \subseteq \tau_2$ for the usual inclusion of topologies. On $\mathbb{R}^d$, we write $|\cdot|$, $B_R$ and $K\subset \subset \mathbb{R}^d$ for the Euclidean norm, the ball with radius $R>0$ and center $0,$ and a compact subset $K$, respectively.

A measure $\zeta$ (always nonnegative) on $\mathcal{B}(\mathbb{R}^d)$ is \textit{locally bounded}, if $\zeta(B_R)<\infty$ for all $R>0$, and a \textit{(sub)probability measure}, if $\zeta(\mathbb{R}^d)$ is (less or) equal to $1$. A curve $t\mapsto \zeta_t$ of locally bounded measures on $\mathcal{B}(\mathbb{R}^d)$ is a \textit{Borel curve}, if $t\mapsto \zeta_t(A)$ is Borel for each $A \subset \subset \mathbb{R}^d$. We denote the spaces of subprobability and probability measures on $\mathbb{R}^d$ with the usual vague and weak topologies of measures, $\tau_v$ and $\tau_w$, by $(\mathcal{SP},\tau_v)$ and $(\mathcal{P},\tau_w)$, respectively. We write $\mathcal{SP}_{\ll}$ and $\mathcal{P}_{\ll}$ for the measures $\zeta$ in $\mathcal{SP}$ and $\mathcal{P}$, respectively, which have a density $d\zeta / dx$ with respect to Lebesgue measure. More generally, for a metric space $X$, we let $\mathcal{P}(X)$ be the set of probability measures on $\mathcal{B}(X)$. A family of Borel functions $h: \mathbb{R}^d \to \mathbb{R}$ is \textit{measure separating}, if $\int h d\zeta_1= \int h d\zeta_2$ for all $h$ implies $\zeta_1=\zeta_2$ for finite Borel measures $\zeta_i$ on $\mathbb{R}^d$.

For a normed space $X$, a Borel measure $\mu$ and $p \in [1,\infty]$, $L^p_{(\loc)}(X;\mu)$ (shortly $L^p_{(\loc)}(X)$, if $X \subseteq \mathbb{R}^d$ and $\mu = dx$) are the usual spaces of (locally) $L^p$-integrable $\mathbb{R}$-valued Borel functions on $X$, with norms on the global spaces denoted by $||\cdot||_{L^p(X)}$. If $X=\mathbb{R}^d$, we write $||\cdot||_p$.
For $s < T$ and a topological space $\mathcal{H}$, $C_{s,T}\mathcal{H}$ denotes the space of continuous functions from $(s,T)$ to $\mathcal{H}$. The usual Sobolev spaces of functions on an open set $U \subseteq \mathbb{R}^d$ are denoted by $W^{m,p}(U)$, with norms $||\cdot||_{W^{m,p}(U)}$.
For $m \in \mathbb{N}_0\cup \{\infty\}$, $C^m_{(c)}(U)$ is the space of continuous functions on $U$ with (compact support and) continuous derivatives of order $m$. For $\gamma \in (0,1]$ and a metric space $X$, set $C_{(\loc)}^\gamma(U,X)$ for the space of (locally) bounded and $\gamma-$Hölder continuous functions from $U$ to $X$, and $C_{(\loc)}^\gamma(U)$, if $X=\mathbb{R}$. If $X$ is normed, $||\cdot||_{C^\gamma(I,X)}$ is the usual Hölder-norm, i.e.
$$||u||_{C^\gamma(U,X)} := \sup_{t \in U}||u(t)||_X+\sup_{t,r\in U}\frac{||u(t)-u(r)||_X}{|t-r|^\gamma}$$

Finally, we recall the definition of the parabolic Sobolev spaces $\mathcal{H}^{p,1}(U,I)$, $\mathbb{H}^{p,1}(U,I)$ from \cite[Ch.6]{FPKE-book15}: For $U \subseteq \mathbb{R}^d$ open and an interval $J \subseteq \mathbb{R}$, set
\begin{align*}
	\mathbb{H}^{p,1}(U,J) := \bigg\{u: J&\times U\to \mathbb{R} \text{ Borel, } u(t,\cdot) \in W^{1,p}(U)  \,dt-a.s.,\, \\&||u||_{\mathbb{H}^{p,1}(U,J)}:= \int_J||u(t,\cdot)||^p_{W^{1,p}(U) }< \infty \bigg\},
\end{align*}
and define $\mathbb{H}^{p,1}_0(U,J)$ similarly, but with $W^{1,p}_0(U)$ (the closure of $C^\infty_c(U)$ with respect to $||\cdot||_{W^{1,p}(U)}$) in place of $W^{1,p}(U)$. The dual of $\mathbb{H}^{p,1}_0(U,J)$ is denoted $\mathbb{H}^{p',-1}_0(U,J)$. Moreover, denote by $\mathcal{H}^{p,1}(U,J)$ the space of elements $u\in \mathbb{H}^{p,1}(U,J)$ with $\partial_t u \in \mathbb{H}^{p,-1}_0(U,J)$, with norm
$$||u||_{\mathcal{H}^{p,1}(U,J)}:= ||u||_{\mathbb{H}^{p,1}(U,J)}+||\partial_t u||_{\mathbb{H}^{p,-1}(U,J)}.$$
Note that compared with our notation for the spaces $W^{m,p}(U)$, we use the reversed order of superindices for $\mathbb{H}^{p,1}(U,J)$ and $\mathcal{H}^{p,1}(U,J)$ to be consistent with \cite{FPKE-book15}.

\subsection{Definitions: Nonlinear FPK equations and solution flows}
Fix $T>0$. Let $\mathcal{S}_0 \subseteq \mathcal{SP}$ and consider coefficients $a = (a_{ij})_{1\leq i,j\leq d}, b = (b_i)_{1 \leq i \leq d}$, 
\begin{equation*}
	a_{ij}, b_i: (0,T)\times \mathcal{S}_0 \times \mathbb{R}^d \to \mathbb{R},\quad 1\leq i,j \leq d.
\end{equation*} 
We always assume $a$ to be symmetric and nonnegative definite. For such coefficients, we consider the operator $\mathcal{L}_{t,\zeta}$ as in \eqref{Def_L_Intro}. Our main objective is the Cauchy problem \eqref{NLFPKE_Intro}-\eqref{Initial-cond_Intro}.

\begin{dfn}\label{Def_NLFPKE-sol}
	A curve of subprobability measures $\mu=(\mu_t)_{t \in (s,T)}$ such that $\mu_t \in \mathcal{S}_0$ $dt$-a.s. is a \textup{solution to the Cauchy problem \eqref{NLFPKE_Intro}-\eqref{Initial-cond_Intro}}, if $t \mapsto \mu_t$ is vaguely continuous, there is a $dt$-version $(\tilde{\mu}_t)_{t \in (s,T)}$ of $\mu$ such that the maps $(t,x) \mapsto a_{ij}(t,\tilde{\mu}_t,x)$ and $(t,x) \mapsto b_i(t,\tilde{\mu}_t,x)$ are Borel, and the following hold:
	\begin{enumerate}
		\item [(i)] $\int_{s}^{T}\int_{K}|a_{ij}(t,\tilde{\mu}_t,x)|+|b_i(t,\tilde{\mu}_t,x)|d\mu_t(x)dt < \infty, \quad \forall K \subset\subset \mathbb{R}^d, 1\leq i,j \leq d.$
		\item[(ii)] For all $\varphi \in C^2_c(\mathbb{R}^d)$ and $t \in (s,T)$
		\begin{equation}\label{Sol-eq_NLFPKE_cont}
			\int_{\mathbb{R}^d}\varphi \,d\mu_{t}- \int_{\mathbb{R}^d}\varphi\, d\nu = \int_{s}^{t}\int_{\mathbb{R}^d}\mathcal{L}_{\tau,\tilde{\mu}_\tau}\varphi \,d\mu_\tau d\tau.
		\end{equation}
	\end{enumerate}
	If $\{\nu,(\mu_t)_{t \in (s,T)}\} \subseteq \mathcal{P}$, we call $\mu$ a \textup{probability solution}.
\end{dfn} 

Probability solutions $t \mapsto \mu_t$ are weakly continuous on $[s,T)$. By \eqref{Sol-eq_NLFPKE_cont} and the global in time integrability (i), each solution extends to a vaguely continuous curve $(\mu_t)_{t \in [s,T]}$ with $\mu_s=\nu$. However, without global in space integrability, $\mu_t(\mathbb{R}^d)=1$ for each $t \in (s,T)$ does not imply $\mu_T(\mathbb{R}^d) = 1$. If $a$ and $b$ are $\mathcal{B}((0,T))\otimes \tau_v \otimes \mathcal{B}(\mathbb{R}^d)$-measurable, one can take $\tilde{\mu}=\mu$ in the above definition.

\begin{rem}\label{1_rem_NL_more-gen}
	\begin{enumerate}
		\item [(i)] It is not necessary to confine the notion of solutions to continuous curves of subprobability measures. More generally, a Borel curve $\mu= (\mu_t)_{t \in (s,T)}$ of locally bounded measures with $\mu_t \in \mathcal{S}_0$ $dt$-a.s. is considered a solution to \eqref{NLFPKE_Intro}-\eqref{Initial-cond_Intro}, provided there is a version $\tilde{\mu}$ of $\mu$ such that $(t,x)\mapsto a_{ij}(t,\tilde{\mu}_t,x)$ and $(t,x) \mapsto b_i(t,\tilde{\mu}_t,x)$ are Borel maps in $L^1_{\loc}((s,T)\times \mathbb{R}^d; \mu_tdt)$ such that for each $\varphi \in C^\infty_c(\mathbb{R}^d)$ there is a set $J_\varphi \subseteq (s,T)$ of full $dt$-measure with
		\begin{equation*}\label{Def_general-sol}
			\int_{\mathbb{R}^d}\varphi\,d\mu_{t}- \int_{\mathbb{R}^d}\varphi\,d\nu = \lim_{r \to s+}\int_{r}^{t}\int_{\mathbb{R}^d}\mathcal{L}_{\tau,\tilde{\mu}_\tau}\varphi\,d\mu_\tau d\tau,\quad \forall t \in J_\varphi,
		\end{equation*}
		compare with \cite[Prop.6.1.2.]{FPKE-book15} and Section 6.7.(iii) of the same reference. Clearly, this general notion coincides with the previous definition, if $(\mu_t)_{t \in (s,T)}$ is vaguely continuous and fulfills (i) of Definition \ref{Def_NLFPKE-sol}.
		Equivalently (see \cite[Prop.6.1.2]{FPKE-book15}), $(\mu_t)_{t \in (s,T)}$ is a solution to \eqref{NLFPKE_Intro}-\eqref{Initial-cond_Intro}, if
		\begin{equation}\label{Sol_time-space-formulation}
			\int_{(s,T)\times\mathbb{R}^d}\partial_t\varphi+\mathcal{L}_{t,\tilde{\mu}_t}\varphi \,d\mu_tdt = 0,\quad \forall \varphi \in C^\infty_c((s,T)\times \mathbb{R}^d),
		\end{equation}
		and for each such $\varphi $, there is $J_\varphi$ as above such that
		\begin{equation}\label{Sol_time-space-formulation_IC}
			\int_{\mathbb{R}^d} \varphi\,d\nu = \lim_{t\to 0+, t \in J_\varphi}\int_{\mathbb{R}^d} \varphi \,d\mu_t.
		\end{equation}
		
		\item[(ii)] It is important to note that solutions $\mu$ to \eqref{NLFPKE_Intro}-\eqref{Initial-cond_Intro} are in particular solutions to the \textup{linear} FPK equation with coefficients $(t,x)\mapsto a(t,\tilde{\mu}_t,x), b(t,\tilde{\mu}_t,x)$. This way, existence and uniqueness results for nonlinear equations are obtained via fixed point arguments \cite{CG19,MS14}.
	\end{enumerate}
\end{rem}
By the following lemma it is no restriction to confine to continuous solutions. We postpone its proof to the appendix.
\begin{lem}\label{Lem_cont-wlog}
	Let $(\mu_t)_{t \in (s,T)}$ be a Borel curve of locally bounded measures with $\mu_t \in \mathcal{S}_0 \subseteq \mathcal{SP}$ $dt$-a.s. such that $(t,x)\mapsto a(t,\tilde{\mu}_t,x)$ and $(t,x)\mapsto b(t,\tilde{\mu}_t,x)$ are Borel maps (for some $dt$-version $\tilde{\mu}$ of $\mu$) in $L^1\big((s,T)\times B_R;\mu_tdt\big)$ for all $R>0$. Assume that \eqref{Sol_time-space-formulation}-\eqref{Sol_time-space-formulation_IC} holds. Then, there exists a unique vaguely continuous version $(\bar{\mu}_t)_{t \in [s,T]}$ of $(\mu_t)_{t \in (s,T)}$, which solves \eqref{NLFPKE_Intro}-\eqref{Initial-cond_Intro} in the sense of Definition \eqref{Def_NLFPKE-sol}, and which is equal to $\nu$ at the initial time $s$. Moreover, for this version \eqref{Sol-eq_NLFPKE_cont} holds for all $t \in [s,T]$.
\end{lem}
For the rest of the paper, we reserve the term \textit{solution} for curves as in Definition \ref{Def_NLFPKE-sol}. 
Finally, we remark that Definition \ref{Def_NLFPKE-sol} includes the case of subprobability solutions to \textit{linear} FPK equations. Indeed, for coefficients $a(t, \zeta, x) := a'(t,x)$ and $b(t,\zeta, x) := b'(t,x)$ (for Borel maps $a', b'$ on $(0,T)\times \mathbb{R}^d$), the linear case is retrieved by choosing $\mathcal{S}_0 = \mathcal{SP}$. 

\paragraph{Solution flows}\label{Subsect_Sol-flow}
We denote by $\mathcal{M}_{s,\nu}$ the set of solutions to \eqref{NLFPKE_Intro}-\eqref{Initial-cond_Intro} with initial condition $(s,\nu) \in [0,T)\times \mathcal{SP}$, and by $\mathcal{M}^1_{s,\nu}$ its subset of probability solutions, if $\nu \in \mathcal{P}$. As we are particularly interested in the singular case of so-called \textit{Nemytskii-type coefficients}, we further set $\mathcal{M}^{\ll}_{s,\nu} := \{(\mu_t)_{t \in (s,T)} \in \mathcal{M}_{s,\nu}: \mu_t \ll dx\, \forall t\}.$
\\

We want to select families $\{\mu^{s,\nu}\}_{(s,\nu) \in [0,T)\times \mathcal{SP}}$ with the flow property \eqref{Flow_intro} such that $\mu^{s,\nu} \in \mathcal{M}_{s,\nu}$. Often, we (need to) restrict to suitable subsets $\mathcal{A}_{s,\nu} \subseteq \mathcal{M}_{s,\nu}$:
\begin{dfn}\label{Def_flow-adm}
	\begin{enumerate}
		\item [(i)] A family $\{\mathcal{A}_{s,\nu}\}_{(s,\nu) \in [0,T)\times \mathcal{SP}}$ with $\mathcal{A}_{s,\nu} \subseteq \mathcal{M}_{s,\nu}$ is \textup{flow-admissible}, if it satisfies the following stability properties for any $0\leq s < r < T$ and $\nu \in \mathcal{SP}$:
		\begin{enumerate}
			\item [(a)] If $(\mu_t)_{t \in (s,T)} \in \mathcal{A}_{s,\nu}$, then $ (\mu_t)_{t \in (r,T)} \in \mathcal{A}_{r,\mu_r}$.
			\item [(b)] If $(\mu_t)_{t \in (s,T)} \in \mathcal{A}_{s,\nu}$ and $(\eta_t)_{t \in (r,T)} \in \mathcal{A}_{r,\mu_r}$, then $\mu \circ_r \eta \in \mathcal{A}_{s,\nu}$, where
			\begin{equation*}
				(\mu \circ_r \eta)_t := \begin{cases}
					\mu_t,& t \in (s,r] \\ \eta_t,& t \in (r,T).
				\end{cases}
			\end{equation*} 
		\end{enumerate}
		\item[(ii)] The set of initial data $\nu \in \mathcal{SP}$ with $\mathcal{A}_{s,\nu} \neq \emptyset$ is denoted by $A_s$, and we call pairs $(s,\nu)$ with $\nu \in A_s$ \textup{admissible}.
	\end{enumerate}
\end{dfn}
Note that a flow-admissible family consists of a set of solutions $\mathcal{A}_{s,\nu} \subseteq \mathcal{M}_{s,\nu}$ for each initial condition $(s,\nu)$, but $A_s$ can be a strict subset of $\mathcal{SP}$. If $A_0$ is non-empty, then each $A_s$ is non-empty by (a) of the previous definition.
\begin{dfn}\label{Def flow}
	\begin{enumerate}
		\item [(i)] 	Let $\{\mathcal{A}_{s,\nu}
		\}_{(s,\nu) \in [0,T)\times \mathcal{SP}}$ be flow-admissible. A \textup{flow} to \eqref{NLFPKE_Intro}-\eqref{Initial-cond_Intro} (with respect to $\{\mathcal{A}_{s,\nu}\}_{(s,\nu) \in [0,T)\times \mathcal{SP}}$) is a family of solutions $\mu^{s,\nu}$ to \eqref{NLFPKE_Intro}-\eqref{Initial-cond_Intro} with $\mu^{s,\nu} \in \mathcal{A}_{s,\nu}$, which fulfills \eqref{Flow_intro} for all $\nu \in A_s$.
	\end{enumerate}
\end{dfn}
\begin{exa}\label{Example_flow-ad families}
	\begin{enumerate}
		\item [(i)] The family $\mathcal{A}^{(1)}_{s,\nu} := \mathcal{M}_{s,\nu}$ is flow-admissible. If $A_s= \mathcal{SP}$ for each $s\in [0,T)$, we call a flow with respect to $\{\mathcal{A}^{(1)}_{s,\nu}\}$ an \textup{entire subprobability flow}.
		\item[(ii)] Likewise, the family
		$$\mathcal{A}^{(2)}_{s,\nu} := \begin{cases}
			\mathcal{M}^1_{s,\nu} &,\text{ if }\nu \in \mathcal{P}\\
			\emptyset&,\text{ else}
		\end{cases}$$
		is flow-admissible, and if $A_s = \mathcal{P}$, we call a flow with respect to $\{\mathcal{A}^{(2)}_{s,\nu}\}$ an \textup{entire probability flow.}
		\item[(iii)] Nemytskii type-coefficients: Let $\mathcal{S}_0 = \mathcal{SP}_{\ll}$ and 
		$$a(t,\zeta,x) = \bar{a}(t,\frac{d\zeta}{dx}(x),x), \,b(t,\zeta,x) = \bar{b}(t,\frac{d\zeta}{dx}(x),x)$$
		for Borel coefficients $\bar{a},\bar{b}$ on $(0,T)\times \mathbb{R}\times \mathbb{R}^d$ (to define $a$ and $b$ pointwise, we consider the version of $d\zeta / dx$ which is $0$ on the complement of its Lebesgue points). Let $\mathfrak{M}\supseteq \mathcal{SP}_{\ll}$. Then, the family $\{\mathcal{A}^{(3)}_{s,\nu}\}$, 
		$$\mathcal{A}^{(3)}_{s,\nu} := \begin{cases}
			\mathcal{M}^\ll_{s,\nu} &,\text{ if }\nu \in \mathfrak{M}\\
			\emptyset&,\text{ else}
		\end{cases}$$
		is flow-admissible.
	\end{enumerate}
\end{exa}

\section{Main results and proofs}
\begin{theorem}\label{Thm-1}
	Let $(\mathcal{H},\tau)$ be a Hausdorff topological space with $\mathcal{H}\subseteq \mathcal{SP}$  and $\tau_v \subseteq \tau$, and let $\{\mathcal{A}_{s,\nu}\}_{(s,\nu) \in [0,T)\times \mathcal{SP}}$ be a flow-admissible family of solution sets to \eqref{NLFPKE_Intro}-\eqref{Initial-cond_Intro} such that $\mathcal{A}_{s,\nu}$ is compact in $C_{s,T}\mathcal{H}$ with respect to the topology of pointwise convergence for each admissible initial condition $(s,\nu)$. Then there exists a flow to \eqref{NLFPKE_Intro}-\eqref{Initial-cond_Intro} with respect to $\{\mathcal{A}_{s,\nu}\}_{(s,\nu) \in [0,T)\times \mathcal{SP}}$.
\end{theorem}
An inspection of the proof of the above result yields the following second main theorem.
\begin{theorem}\label{Thm-2}
	In the situation of Theorem \ref{Thm-1}, the following are equivalent:
	\begin{enumerate}
		\item [(i)] There exists at most one flow with respect to $\{\mathcal{A}_{s,\nu}\}_{(s,\nu) \in [0,T)\times \mathcal{SP}}$.
		\item[(ii)] Solutions in each $\mathcal{A}_{s,\nu}$ are unique.
	\end{enumerate}
\end{theorem}
We formulate Theorem \ref{Thm-1} in a very general way in order to present a unified treatment of global and singular dependencies of solutions on $a$ and $b$, see the applications in Section 4.
\\

The proof of Theorem \ref{Thm-2} will follow as a consequence of the selection method of the proof of Theorem \ref{Thm-1}. Concerning the latter, we need the following prerequisites. 
We set $\mathbb{Q}_s^T := \mathbb{Q}\cap [s,T)$.
\begin{dfn}\label{Def_enumeration}
	\begin{enumerate}
		\item [(i)] We call a bijective map $\xi: \mathbb{N}\times \mathbb{Q}^T_0 \to \mathbb{N}_0$ an \textup{enumeration}. Given such $\xi$ and $k \in \mathbb{N}_0$, we write $(n_k,q_k) :=\xi^{-1}(k)$.
		\item[(ii)]  For $s \in [0,T)$, denote by $(m^s_k)_{k \in \mathbb{N}_0}$ the enumerating sequence of $\mathbb{N}\times \mathbb{Q}^T_s$ with respect to a given enumeration $\xi$, i.e. there exist exactly $k$ elements $(n,q)$ in $\mathbb{N}\times \mathbb{Q}_s^T$ with $\xi(n,q) < m^s_k$.
	\end{enumerate}
\end{dfn}
Note that for $0\leq s < r< T$, the sequence $(m^r_l)_{l \in \mathbb{N}_0}$ is a subsequence of $(m^s_l)_{l \in \mathbb{N}_0}$. Moreover, $C_{s,T}\mathcal{H}$ with the topology of pointwise convergence is Hausdorff, since so is $\mathcal{H}$.\\
\\
\textit{Proof of Theorem \ref{Thm-1}.} Let $H = \{h_n, n \in \mathbb{N}\}\subseteq C_c(\mathbb{R}^d)$ be measure separating and let $\xi$ be an enumeration, for which we use the notation from Definition \ref{Def_enumeration}. Let $(s,\nu) \in [0,T)\times \mathcal{SP}$ be any admissible initial condition and consider
\begin{align*}
	G^{s,\nu}_0: C_{s,T}\mathcal{H} &\to \mathbb{R},\,\, \mu = (\mu_t)_{t \in [s,T)} \mapsto \int_{\mathbb{R}^d} h_{n_{m^s_0}}d\mu_{q_{m^s_0}}, \\
	u^{s,\nu}_0 &:= \underset{\mu \in \mathcal{A}_{s,\nu}}{\text{sup}}	G^{s,\nu}_0(\mu), \\M^{s,\nu}_0&:= ({G^{s,\nu}_0})^{-1}(u^{s,\nu}_0) \cap \mathcal{A}_{s,\nu}.
\end{align*}
Since $\tau_v \subseteq \tau$ and $H \subseteq C_c(\mathbb{R}^d)$, $G^{s,\nu}_0$ is continuous on $C_{s,T}\mathcal{H}$. Furthermore, since $(s,\nu)$ is admissible and $\mathcal{A}_{s,\nu}$ is nonempty and compact by assumption, $M^{s,\nu}_{0}$ is nonempty and compact as well.
Similarly, define iteratively for $k \in \mathbb{N}_0$
\begin{align*}
	G^{s,\nu}_{k+1}: C_{s,T}\mathcal{H} &\to \mathbb{R},\,\, (\mu_t)_{t \in [s,T)} \mapsto \int_{\mathbb{R}^d} h_{n_{m^s_{k+1}}}d\mu_{q_{m^s_{k+1}}}, \\
	u^{s,\nu}_{k+1} &:= \underset{\mu \in M^{s,\nu}_{k}}{\text{sup}} G^{s,\nu}_{k+1}(\mu), \\M^{s,\nu}_{k+1}&:= ({G^{s,\nu}_{k+1}})^{-1}(u^{s,\nu}_{k+1}) \cap M^{s,\nu}_{k}.
\end{align*}Then the same assertions as for $G^{s,\nu}_0$ and $M_0^{s,\nu}$ are true for $G^{s,\nu}_{k+1}$ and $M_{k+1}^{s,\nu}$. Since 
$M^{s,\nu}_{k+1} \subseteq M^{s,\nu}_{k}$ and $C_{s,T}\mathcal{H}$ is Hausdorff, we obtain
$$M^{s,\nu} := \bigcap_{k \geq 0} M^{s,\nu}_{k} \neq \emptyset.$$
Now assume $\mu^{(i)} = (\mu^{(i)}_t)_{t \in [s,T)} \in M^{s,\nu}$ for $i\in \{1,2\}$. By construction, this implies 
\begin{equation*}
	\int_{\mathbb{R}^d} h_{n_{m^s_k}}d\mu^{(1)}_{q_{m^s_k}} = \int_{\mathbb{R}^d} h_{n_{m^s_k}}d\mu^{(2)}_{q_{m^s_k}},\quad k \in \mathbb{N}_0.
\end{equation*}Since $\{(n_{m^s_k},q_{m^s_k}), \, k \in \mathbb{N}_0\} =  \mathbb{N}\times \mathbb{Q}^T_s$, this yields $\int h_n d\mu^{(1)}_q = \int h_n d\mu^{(2)}_q$ for all $(n,q)\in \mathbb{N}\times \mathbb{Q}^T_s$ and hence $\mu^{(1)}_q = \mu^{(2)}_q$ for all $q \in \mathbb{Q}_s^T$, because $H$ is measure separating. Since both $\mu^{(1)}$ and $\mu^{(2)}$ are continuous in the Hausdorff space $\mathcal{H}$, $\mu^{(1)} = \mu^{(2)}$ follows. Consequently, $M^{s,\nu} \subseteq \mathcal{A}_{s,\nu}$ is a singleton, i.e. $\mathcal{M}^{s,\nu}  =\{\mu^{s,\nu}\}$. 

It remains to show that the family $\{\mu^{s,\nu}\}$ is a flow. To this end, let $(s,\nu)$ be admissible, and fix $0\leq s < r < t< T$. Consider the admissible initial condition $(r,\mu^{s,\nu}_r)$ and let $\gamma = (\gamma_t)_{t \in [r,T)}\in M^{r,\mu^{s,\nu}_r}$ be the corresponding uniquely selected solution from the first part of the proof, i.e. with our notation $\gamma = \mu^{r,\mu^{s,\nu}_r}$. We need to show
\begin{equation}\label{Eq for flow}
	\gamma_t=\mu^{s,\nu}_t,\quad \forall\,t\in [r,T).
\end{equation}
Set $\eta := \mu^{s,\nu}\circ_r \gamma \in \mathcal{A}_{s,\nu}$. Due to the iterative maximizing selection procedure of the first part of the proof, we have
\begin{equation}\label{important_4}
	\int_{\mathbb{R}^d} h_{n_{m^s_0}}d\mu^{s,\nu}_{q_{m^s_0}} \geq \int_{\mathbb{R}^d} h_{n_{m^s_0}}d\eta_{q_{m^s_0}}. 
\end{equation}If $q_{m^s_0} \in [s,r)$, then $\eta_{q_{m^s_0}}=\mu^{s,\nu}_{q_{m^s_0}}$ and we have equality in (\ref{important_4}). If $q_{m^s_0} \in \,[r,T)$, then $q_{m^s_0}=q_{m^r_0}$ and by the characterizing property of $\gamma$ in $\mathcal{A}_{r,\mu^{s,\nu}_r}$, and since $(\mu^{s,\nu}_t)_{t \in [r,T)} \in \mathcal{A}_{r,\mu^{s,\nu}_r}$, we obtain
$$\int_{\mathbb{R}^d} h_{n_{m^s_0}}d\mu^{s,\nu}_{q_{m^s_0}} \leq \int_{\mathbb{R}^d} h_{n_{m^s_0}}d\gamma_{q_{m^s_0}} =\int_{\mathbb{R}^d} h_{n_{m^s_0}}d\eta_{q_{m^s_0}},$$
and hence we have equality in (\ref{important_4}) in any case.  Next, consider $m^s_1$: since (\ref{important_4}) is an equality, both $(\mu^{s,\nu}_t)_{t \in [s,T)}$ and $(\eta_t)_{t \in [s,T)}$ belong to $M_0^{s,\nu}$. Hence, using the characterization of $\mu^{s,\nu}$ again, we obtain
\begin{equation}\label{important_6}
	\int_{\mathbb{R}^d} h_{n_{m^s_1}}d\mu^{s,\nu}_{q_{m^s_1}} \geq \int_{\mathbb{R}^d} h_{n_{m^s_1}}d\eta_{q_{m^s_1}},
\end{equation}clearly with equality if $q_{m^s_1} \in [s,r)$. If $q_{m^s_1} \in \, [r,T)$ and $q_{m^s_0} \in [s,r)$, then $m^s_1 = m^r_0$, i.e.
\begin{equation}\label{important_7}
	\int_{\mathbb{R}^d} h_{n_{m^s_1}}d\mu^{s,\nu}_{q_{m^s_1}} \leq \int_{\mathbb{R}^d} h_{n_{m^s_1}}d\gamma_{q_{m^s_1}} = \int_{\mathbb{R}^d} h_{n_{m^s_1}}d\eta_{q_{m^s_1}}
\end{equation}by the characterizing property of $\gamma$, which gives equality in (\ref{important_6}). If $q_{m^s_0}, q_{m^s_1} \in\, [r,T)$, then $m_0^s=m_0^r$, $m^s_1=m^r_1$ and both $\mu^{s,\nu}$ and $\gamma$ are in $M_{0}^{r,\mu^{s,\nu}_r}$, which also gives (\ref{important_7}). Hence, equality in (\ref{important_6}) holds in any case. Iterating this procedure yields
\begin{equation*}
	\int_{\mathbb{R}^d} h_{n_{m^s_k}}d\mu^{s,\nu}_{q_{m^s_k}} = \int_{\mathbb{R}^d} h_{n_{m^s_k}}d\eta_{q_{m^s_k}},\quad \forall \, k \in \mathbb{N}_0,
\end{equation*}
and hence, since $H$ is measure separating,
$$\mu^{s,\nu}_q = \eta_q,\quad \forall q\in \mathbb{Q}^T_s, $$so in particular $\mu^{s,\nu}_q = \eta_q =  \gamma_q$ for all $q \in \mathbb{Q}^T_r$.
Since both curves are continuous, we obtain (\ref{Eq for flow}).
\qed
\begin{rem}
	The proof works for any separable measure separating family from $C_c(\mathbb{R}^d)$ and for any dense countable subset of $[s,T)$ instead of $\mathbb{Q}_s^T$. In general, the selected flow depends on these choices.
\end{rem}
\textit{Proof of Theorem \ref{Thm-2}.}
For the nontrivial implication of the assertion, assume there is an admissible initial condition $(s',\nu') \in [0,T)\times \mathcal{SP}$ with $|\mathcal{A}_{s',\nu'}| \geq 2$. As mentioned in the previous remark, we may choose an enumeration $\xi$ and a family of measure separating functions $H= \{h_n, n \in \mathbb{N}\} \subseteq C_c(\mathbb{R}^d)$ with $H=-H$. Consider the flow $\{\mu^{s,\nu}\}$ with $(s,\nu)$ running through all admissible initial conditions, constructed as in the proof of Theorem \ref{Thm-1} subject to this $H$ and $\xi$. By assumption, there exists $\gamma \in \mathcal{A}_{s',\nu'}$ with $\mu^{s',\nu'} \neq \gamma$, and since both curves are continuous, there is $q \in \mathbb{Q}_{s'}^T$ such that $\mu^{s',\nu'}_{q} \neq \gamma_{q}$. Thus, considering $-h$ instead of $h$ if necessary, there is $h \in H$ such that
\begin{equation}\label{ineq1}
	\int_{\mathbb{R}^d}h\,d\gamma_{q} > \int_{\mathbb{R}^d}h\,d\mu^{s',\nu'}_{q}.
\end{equation}
Now consider a new enumeration $\xi'$ such that according to $\xi'$ we have $(h_{n_0},q_{n_0}) = (h,q)$, and denote the flow subject to $H$ and $\xi'$ by $\{\eta^{s,\nu}\}$ (the sets of admissible initial conditions remain unchanged). Selecting as in the proof of Theorem \ref{Thm-1} gives
\begin{equation*}
	\int_{\mathbb{R}^d}h \,d\eta^{s',\nu'}_{q} = \underset{\mu \in \mathcal{A}_{s',\nu'}}{\sup} \bigg(\int_{\mathbb{R}^d}h \,d\mu_{q}\bigg).
\end{equation*}
Therefore, taking into account (\ref{ineq1}), we conclude
\begin{equation*}
	\int_{\mathbb{R}^d}h \,d\eta^{s',\nu'}_{q} \geq \int_{\mathbb{R}^d}h\, d\gamma_{q} > \int_{\mathbb{R}^d}h\,d\mu^{s',\nu'}_{q}.
\end{equation*}Hence $\eta^{s',\nu'} \neq \mu^{s',\nu'}$, which contradicts (i) and finishes the proof. 
\qed
\begin{rem}
	As mentioned in the introduction, for the case of linear equations, i.e. for $a$ and $b$ independent of $\zeta \in \mathcal{SP}$, the above results particularly improve Theorems 3.2. and 3.16. from \cite{Rehmeier_Flow-JEE}, where we exploited the connection between \eqref{NLFPKE_Intro} and the associated martingale problem via the \textup{superposition principle} in order to obtain compactness of solutions to \eqref{NLFPKE_Intro}-\eqref{Initial-cond_Intro} via compactness of solutions to the martingale problem. The latter is, however, only known for linear cases under rather restrictive assumptions on $a$ and $b$. Moreover, in \cite{Rehmeier_Flow-JEE}, we only considered entire probability selections in the sense of Example \ref{Example_flow-ad families} (ii).
\end{rem}

\section{Applications}
Note that in any of the forthcoming applications also Theorem \ref{Thm-2} holds. For the convenience of the reader, we state a sufficiently general version of the Arzelà-Ascoli theorem, which we use several times below. For topological spaces $I$ and $Y$, denote by $C(I,Y)$ the set of continuous functions from $I$ to $Y$.
\begin{prop}[Arzelà-Ascoli theorem, Thm.47.1 \cite{Munkres_top_book}]\label{AA}
	Let $I$ be an interval and $(Y,d)$ a metric space. Then, $\mathcal{F}\subseteq C(I,Y)$ is relatively compact in the compact-open topology if and only if $\mathcal{F}$ is pointwise relatively compact and equicontinuous, i.e. if
	\begin{enumerate}
		\item [(i)]$\{f(t), f \in \mathcal{F}\}$ is relatively compact in $Y$ for all $t \in I$
		\item[(ii)] For all $t\in I$ and $\varepsilon >0$ there is $\delta>0$ such that $$r \in I, |t-r|< \delta \implies \sup_{f \in \mathcal{F}}d(f(t),f(r)) < \varepsilon.$$
	\end{enumerate}
\end{prop}
Recall that the compact-open topology on $C(I,Y)$ coincides with the topology of locally uniform convergence, which is the topology of uniform convergence, if $I$ is compact.

\subsection{Linear equations}\label{Subsection_lin-appl}
Here we consider coefficients of \textit{linear} FPK equations, i.e. $a$ and $b$ are independent of $\zeta \in \mathcal{S}_0$.
\subsubsection{Entire subprobability flows} Suppose the Borel coefficients $a_{ij},b_i: (0,T)\times \mathbb{R}^d \to \mathbb{R}$, $1\leq i,j \leq d$, satisfy \\
\\
\textbf{Assumption A1.}
\begin{enumerate}
	\item [(A1.i)] $\int_{0}^{T}\sup_{x \in K}\big(|a_{ij}(t,x)|+|b_i(t,x)|\big)dt < \infty \,\,\forall\, K \subset \subset \mathbb{R}^d$
	\item[(A1.ii)] $x \mapsto a_{ij}(t,x)$ and $x \mapsto b_i(t,x)$ are continuous for $dt$-a.a. $t \in (0,T)$.
\end{enumerate}
\begin{prop}\label{Prop_entire-flow-SP}
	Suppose that Assumption A1 is fulfilled and that $\mathcal{M}_{s,\nu}$ is nonempty for each $(s,\nu) \in [0,T)\times \mathcal{SP}$. Then, there exists an entire subprobability flow $\{\mu^{s,\nu}\}_{(s,\nu) \in [0,T)\times \mathcal{SP}}$ for \eqref{NLFPKE_Intro}-\eqref{Initial-cond_Intro}.
\end{prop}
Concerning the proof, appealing to Theorem \ref{Thm-1}, we choose 
$$(\mathcal{H},\tau) = (\mathcal{SP},\tau_v)\text{ and }\mathcal{A}_{s,\nu} = \mathcal{M}_{s,\nu}.$$
$(\mathcal{SP},\tau_v)$ is Polish and compact. Due to (A1.i), any solution $\mu$ extends uniquely from $(s,T)$ to $[s,T]$ with $\mu_s= \nu$. Hence, here we can replace $C_{s,T}\mathcal{SP}$ by $C([s,T],\mathcal{SP})$. We consider the latter space with the topology of uniform convergence, which is independent of the choice of compatible metric $d_v$ on $\mathcal{SP}$. Indeed, this topology coincides with the compact-open topology on $C([s,T],\mathcal{SP})$. Since the topology of uniform convergence is stronger than the topology of pointwise convergence, Proposition \ref{Prop_entire-flow-SP} follows from the following lemma and Theorem \ref{Thm-1}.
\begin{lem}\label{Lem_compactness_A1-case}
	Consider $C([s,T],\mathcal{SP})$ with the topology of uniform convergence. If Assumption A1 holds, then $\mathcal{M}_{s,\nu}$ is compact in $C([s,T],\mathcal{SP})$.
\end{lem}
\begin{proof}
	Let $(s,\nu) \in [0,T)\times \mathcal{SP}$. In order to use Proposition \ref{AA}, we show pointwise relative compactness, equicontinuity and closedness of $\mathcal{M}_{s,\nu}$. Pointwise relative compactness, i.e. relative compactness of $\{\mu_t: \mu \in \mathcal{M}_{s,\nu}\} \subseteq \mathcal{SP}$ for each $t$, holds since $(\mathcal{SP},\tau_v)$ is compact. 
	Concerning closedness, let $\mu^{(n)} = (\mu^{(n)}_t)_{t \in [s,T]},n \geq 1,$ be a converging sequence in $\mathcal{M}_{s,\nu}$ with limit $\mu = (\mu_t)_{t \in [s,T]}\in C([s,T],\mathcal{SP})$ and let $\varphi \in C_c^2(\mathbb{R}^d)$. Clearly, 
	$$\int_{\mathbb{R}^d} \varphi\,d\mu^{(n)}_t \underset{n \to \infty}{\longrightarrow} \int_{\mathbb{R}^d}\varphi\,d\mu_t, \quad \forall \,t \in [s,T],$$
	and in particular $\mu_s = \nu$. Furthermore, due to (A1.ii), we have $\mathcal{L}_t\varphi \in C_c(\mathbb{R}^d)$ $dt$-a.s. Consequently, 
	$$\int_{\mathbb{R}^d}\mathcal{L}_t\varphi\,d\mu^{(n)}_t \underset{n \to \infty}{\longrightarrow} \int_{\mathbb{R}^d}\mathcal{L}_t\varphi\,d\mu_t\quad dt\text{-a.s.},$$
	and by (A1.i), Lebesgue's dominated convergence theorem gives
	$$\int_s^t\int_{\mathbb{R}^d}\mathcal{L}_\tau\varphi\,d\mu^{(n)}_\tau d\tau \underset{n \to \infty}{\longrightarrow} \int_s^t\int_{\mathbb{R}^d}\mathcal{L}_\tau\varphi\,d\mu_\tau d\tau.$$
	Therefore, $\mu \in \mathcal{M}_{s,\nu}$. 
	Finally, in general equicontinuity of a subset $\mathcal{C} \subset C([s,T],\mathcal{SP})$ is not independent of the particular $\tau_v$-compatible metric $d_v$ on $\mathcal{SP}$. However, since the uniform topology on $C([s,T],\mathcal{SP})$ is independent of $d_v$ and since pointwise relative compactness and closedness are topological properties independent of $d_v$, it follows from the previous parts of the proof and the Arzelà-Ascoli theorem that equicontinuity of $\mathcal{M}_{s,\nu}$
	holds either for every or none $\tau_v$-compatible metric $d_v$. Thanks to this observation, we make the following convenient choice for $d_v$:
	\begin{equation*}
		d_v(\zeta_1,\zeta_2) := \sum_{l \geq 1}2^{-l}C_l^{-1}\bigg[\bigg|\int_{\mathbb{R}^d} f_l d\zeta_1 - \int_{\mathbb{R}^d} f_l d\zeta_2\bigg|\wedge 1\bigg],
	\end{equation*}
	where $\{f_l, l\in \mathbb{N}\}=:\mathcal{F} \subseteq C_c^2(\mathbb{R}^d)$ is arbitrary but fixed and consists of nontrivial elements such that the closure of $\mathcal{F}$ with respect to uniform convergence is $C_c(\mathbb{R}^d)$. Moreover, we set
	$$C_l:=1+D_l, \quad D_l:=(d^2+d)\underset{1\leq i,j \leq d}{\max} \{||\partial_if_l||_{\infty},||\partial_{ij}f_l||_{\infty}\}.$$
	By the choice of $\mathcal{F}$, $d_v$ is $\tau_v$-compatible. We obtain for each $\mu = (\mu_t)_{t \in [s,T]} \in \mathcal{M}_{s,\nu}$ and arbitrary $s\leq t_1\leq t_2 \leq T$:
	\begin{align}\label{1_calc1}
		\notag d_v(\mu_{t_1},\mu_{t_2}) &\leq  \sum_{l \geq 1}2^{-l}C_l^{-1}\bigg[\int_{t_1}^{t_2}\int_{\mathbb{R}^d}|\mathcal{L}_tf_l|d\mu_tdt \wedge 1 \bigg] \\& \leq \sum_{l \geq 1}2^{-l}C_l^{-1}D_l\bigg[\int_{t_1}^{t_2}\max_{1\leq i,j \leq d}\sup_{x \in K_l}\big(|a_{ij}(t,x)|+|b_i(t,x)|\big)dt\wedge 1\bigg] \\& \notag
		\leq \sum_{l \geq 1}2^{-l}\bigg[\int_{t_1}^{t_2}\max_{1\leq i,j \leq d}\sup_{x \in K_l}\big(|a_{ij}(t,x)|+|b_i(t,x)|\big)dt\wedge 1\bigg],
	\end{align}
	with $K_l := \supp f_l$. Hence, using (A1.i), for any $\varepsilon >0$ and $L \geq 1$, there is $\delta >0$ independent of $\mu \in \mathcal{M}_{s,\nu}$ such that
	$$t_1,t_2 \in [s,T], \,|t_1-t_2| \leq \delta \implies \max_{l \leq L}\int_{t_1\wedge t_2}^{t_1\vee t_2}\max_{1\leq i,j \leq d}\underset{x \in K_l}{\text{sup}}\big(|a_{ij}(t,x)|+|b_i(t,x)|\big)dt < \varepsilon/2.$$
	Consequently, $\mathcal{M}_{s,\nu}$ is (uniformly) equicontinuous with respect to any $\tau_v$-compatible metric on $\mathcal{SP}$, which completes the proof. 
\end{proof}
\begin{rem}\label{Rem_union-precompt_as-well}
	The estimates \eqref{1_calc1} are independent of the initial measure $\nu$, so we obtain even relative compactness of $\cup_{\nu \in \mathcal{SP}}\mathcal{M}_{s,\nu}$.
\end{rem}
\begin{rem}\label{1_rem_aux1}
	Proposition \ref{Prop_entire-flow-SP} applies in the situation of \cite[Thm.6.7.3]{FPKE-book15}. Indeed, in that case our Assumption A1 holds with $c \equiv 0$ from \cite{FPKE-book15}, and elements $ \mu \in \mathcal{M}_{\nu}$ as in \cite[Thm.6.7.3]{FPKE-book15} are curves of subprobability measures $dt$-a.s., if the initial datum $\nu$ is in $\mathcal{SP}$ (note that the existence result of the cited source is stated for $\nu \in \mathcal{P}$ only, but obviously extends to $\nu \in \mathcal{SP}$). For each such solution curve $\mu$, Lemma \ref{Lem_cont-wlog} (applied to the case of linear equations) gives a vaguely continuous version, which shows that each initial condition is admissible. 
\end{rem}

\subsubsection{Entire probability flows}\label{Subsubsection-prob-flow-lin}
Here we suppose the Borel coefficients $a_{ij},b_i: (0,T)\times \mathbb{R}^d\to \mathbb{R}$, $1 \leq i,j \leq d$, satisfy \\
\\
\textbf{Assumption A2.}
\begin{enumerate}
	\item [(A2.i)] $\int_{0}^{T}\sup_{x \in \mathbb{R}^d}\big(|a_{ij}(t,x)|+|b_i(t,x)|\big)dt < \infty$.
	\item [(A2.ii)] $x \mapsto a_{ij}(t,x)$ and $x \mapsto b_i(t,x)$ are continuous for $dt$-a.a. $t \in (0,T)$.
\end{enumerate}
In this case, for $\nu \in \mathcal{P}$ any solution $\mu \in \mathcal{M}_{s,\nu}$ extends to a curve $(\mu_t)_{t \in [s,T]}\subseteq \mathcal{P}$, i.e. in particular $\mathcal{M}_{s,\nu} = \mathcal{M}^1_{s,\nu}$. Indeed, considering \eqref{Sol-eq_NLFPKE_cont} for $\mu$ and a nonnegative sequence $\{\varphi_n\}_{n \geq 1} \subseteq C^2_c(\mathbb{R}^d)$, which increases pointwise to $1$ such that $\sup_{n \geq 1}||\varphi_n||_{C^2_b} < \infty$ and $\varphi_n =1$ on $(-n,n)$, we obtain $\mu_t(\mathbb{R}^d) = \nu(\mathbb{R}^d)$ for any $t \in [s,T]$. In particular, $t\mapsto \mu_t$ is weakly continuous.

Again, we consider $(\mathcal{H},\tau) = (\mathcal{SP},\tau_v)$ and $C([s,T],\mathcal{SP})$ instead of $C_{s,T}\mathcal{SP}$. We select from the sets
\begin{equation}\label{Def A}
	\mathcal{A}_{s,\nu} := \begin{cases}
		\mathcal{M}_{s,\nu}&, \text{ if }\nu \in \mathcal{P} \\
		\emptyset&, \text{ if } \nu \in \mathcal{SP}\backslash \mathcal{P}.
	\end{cases}
\end{equation}

\begin{prop}\label{Prop_entire_prob}
	Suppose that Assumption A2 is fulfilled and that $\mathcal{M}_{s,\nu} = \mathcal{M}_{s,\nu}^1$ is nonempty for each $(s,\nu) \in [0,T)\times \mathcal{P}$. Then there exists an entire probability flow for \eqref{NLFPKE_Intro}-\eqref{Initial-cond_Intro}.
\end{prop}
\begin{proof}
	Since $\mathcal{M}_{s,\nu} = \mathcal{M}_{s,\nu}^1$ for $\nu \in \mathcal{P}$, the family in (\ref{Def A}) is flow-admissible. Compactness of $\mathcal{A}_{s,\nu} \subseteq C([s,T],\mathcal{SP})$ follows as in the proof of Proposition \ref{Prop_entire-flow-SP}. Hence, Theorem \ref{Thm-1} applies and gives the assertion.
\end{proof}
In \cite{Rehmeier_Flow-JEE}, Proposition \ref{Prop_entire_prob} was proven via the connection of \eqref{NLFPKE_Intro}-\eqref{Initial-cond_Intro} to the associated martingale problem, under the slightly stronger assumption that $a$ and $b$ are bounded on $(0,T)\times \mathbb{R}^d$.

\subsubsection{Flows subject to Lyapunov functions}
Suppose $a$ and $b$ are Borel maps and fulfill Assumption A1. In order to select a flow of probability solutions in the case of unbounded coefficients, a control on $a$ and $b$ via \textit{Lyapunov functions} turns out to be helpful. Let $\psi \in C^2(\mathbb{R}^d)$ be nonnegative and \textit{compact}, i.e. $\{\psi \leq c\} \subset \subset \mathbb{R}^d$ for any $c \geq0$. Assume there is $C>0$ such that
\begin{equation}\label{Lypunov ineq.}
	\mathcal{L}_t\psi(x) \leq C+C\psi(x)\,\,\,\, dxdt\text{-a.s. in }(0,T)\times \mathbb{R}^d.
\end{equation} 
Such $\psi$ is often called a \textit{Lyapunov function}. For a Lyapunov function $\psi$, we set
$$\mathcal{P}_\psi := \{\zeta \in \mathcal{P}: \psi \in L^1(\mathbb{R}^d;\zeta)\}.$$
As a simple consequence of \cite[Lemma 2.2.]{BRDP08}, we have: for any initial condition $(s,\nu)\in [0,T)\times \mathcal{P}_\psi$ and each $t \in [s,T)$, there is a number $0<C(s,\nu,t,\psi) < \infty$ such that 
\begin{equation}\label{Tightness_Lyapunov-fct}
	\sup_{\mu \in \mathcal{M}^1_{s,\nu}}\int_{\mathbb{R}^d}\psi \,d\mu_t \leq C(s,\nu,t,\psi).
\end{equation}
We aim to select a flow from the sets
\begin{equation}\label{family of sets}
	\mathcal{A}_{s,\nu} = 
	\begin{cases}
		\mathcal{M}^1_{s,\nu}&, \text{ if }\nu \in \mathcal{P}_{\psi}\\
		\emptyset&, \text{ else}
	\end{cases},
\end{equation}
which is flow-admissible by \eqref{Tightness_Lyapunov-fct}.
With regard to Theorem \ref{Thm-1}, we let $(\mathcal{H},\tau) = (\mathcal{P},\tau_w)$.
\begin{prop}\label{Lyap flow prop}
	Suppose $a$ and $b$ fulfill Assumption A1 and there exists a nonnegative compact function $\psi$ such that (\ref{Lypunov ineq.}) holds. If the sets $\mathcal{M}^1_{s,\nu}$ are nonempty for each initial condition $(s,\nu) \in [0,T)\times \mathcal{P}_\psi$, then there exists a flow of probability solutions with respect to the family defined in (\ref{family of sets}).
\end{prop}

\begin{proof}
	By Theorem \ref{Thm-1}, for each $\nu \in \mathcal{P}_{\psi}$ it suffices to show compactness of $\mathcal{A}_{s,\nu}$ in $C([s,T),\mathcal{P})$ (with the compact-open topology). To this end, we again evoke Proposition \ref{AA}. Due to \eqref{Tightness_Lyapunov-fct}, $\{\mu_t: \mu \in \mathcal{M}^1_{s,\nu}\}$ is tight, hence relatively compact in $(\mathcal{P},\tau_w)$ for each $t \in [s,T)$. Concerning closedness, it follows as in the proof of Lemma \ref{Lem_compactness_A1-case} that the limit of any converging sequence in $\mathcal{M}^1_{s,\nu}$ belongs to $\mathcal{M}^1_{s,\nu}$ as well. Finally, (uniform) equicontinuity of $\mathcal{M}^1_{s,\nu}$ can be proven exactly as in the proof of Lemma \ref{Lem_compactness_A1-case}. Indeed, the metric $d_v$ used in that proof restricted to $\mathcal{P}$ metrizes $\tau_w$, since $\tau_w$ coincides with $\tau_v$ restricted to $\mathcal{P}$. 
\end{proof}

\begin{rem}
	Consider the situation of \cite[Theorem 3.1]{BRDP08} and assume additionally that $b_i$, $1\leq i \leq d$, satisfies Assumption A1. In this case, $a_{ij}$, $1 \leq i,j \leq d$, has a version which fulfills Assumption A1. Indeed, by \cite[(C1),(C2)]{BRDP08} and Sobolev embedding, each $a(t,\cdot)$ has a Hölder continuous version, which is bounded on balls in $\mathbb{R}^d$, independently of $t$. Then, for any $(s,\nu) \in [0,T)\times \mathcal{P}$, \cite[Thm.3.1]{BRDP08} yields the existence of a weakly continuous probability solution $(\mu^{s,\nu}_t)_{t \in [s,T)}$, so that Proposition \ref{Lyap flow prop} applies.
\end{rem}

\subsection{Nonlinear equations}\label{Subsection_nonlin-appl}
We turn to nonlinear equations of type \eqref{NLFPKE_Intro}-\eqref{Initial-cond_Intro}. We consider two prototypes of equations: First, we treat the case of coefficients $(t,\zeta,x)\mapsto a(t,\zeta,x)$ and $(t,\zeta,x)\mapsto b(t,\zeta,x)$ which are continuous in $\zeta$ with respect to the vague (or weak) topology on $\mathcal{SP}$. In this case, $(a_{ij}(t,x))_{i,j \leq d}$ can be be degenerate. Secondly, we treat the case where $b(t,\zeta,x)$ depends on $\zeta$ via its density at $x$, i.e. $b$ is of \textit{Nemytskii-type}. In this case, $b$ is typically not continuous with respect to $\tau_v$ or $\tau_w$.
\subsubsection{Global dependence}
Set $\mathcal{S}_0 = \mathcal{SP}$ and consider $\mathcal{B}((0,T))\otimes \tau_v\otimes \mathcal{B}(\mathbb{R}^d)$-measurable coefficients $a_{ij}, b_i: (0,T)\times \mathcal{SP}\times \mathbb{R}^d \to \mathbb{R}$, such that the following assumption holds for each $1 \leq i,j \leq d$.
\\
\\
\textbf{Assumption B1.}
\begin{enumerate}
	\item[(B1.i)] $\int_{0}^{T}\sup_{(\zeta,x) \in \mathcal{SP}\times K}\big(|a_{ij}(t,\zeta,x)|+|b_i(t,\zeta,x)|\big)dt < \infty$ $\forall K \subset \subset \mathbb{R}^d$.
	\item [(B1.ii)] $x \mapsto a_{ij}(t,\zeta,x)$ and $b_i(t,\zeta,x)$ are continuous for each $\zeta \in \mathcal{SP}$ and $dt$-a.a. $t\in(0,T)$.
	\item[(B1.iii)] If $\zeta_n \longrightarrow \zeta$ vaguely, then $a_{ij}(t,\zeta_n,x)\longrightarrow a_{ij}(t,\zeta,x)$ and $b_{i}(t,\zeta_n,x)\longrightarrow b_{i}(t,\zeta,x)$ locally uniformly in $x$ for each $t \in (0,T)$.
\end{enumerate}
(B1.i) and (B1.ii) are comparable to Assumption A1, and the strong additional assumption (B1.iii) is used to show closedness of $\mathcal{M}_{s,\nu}$, which is more delicate compared to the linear case. Set, for each $(s,\nu)\in [0,T)\times \mathcal{SP}$,
\begin{equation}\label{Choice-H-and-A_first-nonlin}
	(\mathcal{H},\tau) = (\mathcal{SP},\tau_v)\text{ and }\mathcal{A}_{s,\nu} = \mathcal{M}_{s,\nu}.
\end{equation}
As in the linear case, solutions extend from $(s,T)$ to $[s,T]$ and we replace $C_{s,T}\mathcal{SP}$ by $C([s,T],\mathcal{SP})$. Since $a$ and $b$ are measurable with respect to $\tau_v$, one can always choose $\mu = \tilde{\mu}$ in Definition \eqref{Def_NLFPKE-sol}.
\begin{prop}\label{1_prop_NL_entire-SP}
	Suppose Assumption B1 is fulfilled and that $\mathcal{M}_{s,\nu}$ is nonempty for each $(s,\nu) \in [0,T)\times \mathcal{SP}$. Then there exists an entire subprobability flow $\{\mu^{s,\nu}\}_{(s,\nu) \in [0,T)\times \mathcal{SP}}$ for \eqref{NLFPKE_Intro}-\eqref{Initial-cond_Intro}. 
\end{prop}
\begin{proof}\label{Prop_global-nonlinear-case-1}
	Again, we evoke the Arzelà-Ascoli theorem \ref{AA} and appeal to Theorem \ref{Thm-1}. Relative compactness of $\{\mu_t: \mu \in \mathcal{M}_{s,\nu}\} \subseteq \mathcal{SP}$ is obvious, since $(\mathcal{SP},\tau_v)$ is compact.
	Equicontinuity of $\mathcal{M}_{s,\nu}$ can be proven as in the proof of Lemma \ref{Lem_compactness_A1-case}, using (B1.i) instead of (A1.i). Therefore, $\mathcal{M}_{s,\nu} \subseteq C([s,T],\mathcal{SP})$ is relatively compact.
	For closedness, assume $\mu^{(n)} = (\mu^{(n)}_t)_{t \in [s,T]}$ converges to $\mu = (\mu_t)_{t \in [s,T]}$ in $C([s,T],\mathcal{SP})$. We need to prove
	\begin{equation}\label{1_aux2}
		\int_{s}^{t}\int_{\mathbb{R}^d}\mathcal{L}_{r,\mu^{(n)}_r}\varphi \,d\mu^{(n)}_r dr \underset{n \to \infty}{\longrightarrow} \int_{s}^{t}\int_{\mathbb{R}^d}\mathcal{L}_{r,\mu_r}\varphi \,d\mu_rdr
	\end{equation}
	for each $\varphi \in C^\infty_c(\mathbb{R}^d)$ and $t \in (s,T)$. This can, for example, be realized by rewriting
	$$\int_{\mathbb{R}^d}\mathcal{L}_{r,\mu^{(n)}_r}\varphi \,d\mu^{(n)}_r = {}_{C_c^*}\big \langle \mu^{(n)}_r, \mathcal{L}_{r,\mu^{(n)}_r}\varphi\big \rangle_{C_c},$$
	where ${}_{C_c^*}\big \langle \mu,f \big \rangle_{C_c}$ denotes the dual pairing of $f \in (C_c(\mathbb{R}^d),||\cdot||_\infty)$ and a bounded Borel measure $\mu$. Since $\tau_v$ coincides with the $\text{weak-}^*$ topology on the topological dual space of $C_c(\mathbb{R}^d)$, and since assumptions (B1.ii) and (B1.iii) yield $\mathcal{L}_{r,\mu^{(n)}_r}\varphi \longrightarrow \mathcal{L}_{r,\mu_r}\varphi$ in $(C_c(\mathbb{R}^d),||\cdot||_\infty)$, it follows that
	$${}_{C_c^*}\big \langle \mu_r^{(n)}, \mathcal{L}_{r,\mu^{(n)}_r}\varphi\big \rangle_{C_c} \longrightarrow {}_{C_c^*}\big \langle \mu_r, \mathcal{L}_{r,\mu_r}\varphi\big \rangle_{C_c}.$$
	From here, \eqref{1_aux2} follows by (B1.i) and Lebesgue's dominated convergence theorem.
\end{proof}
We continue by showing that one can slightly weaken the restrictive assumption (B1.iii) in the case of \textit{globally} bounded coefficients. Suppose that, instead of Assumption B1, $a$ and $b$ fulfill
\\
\\
\textbf{Assumption B2.} 
\begin{enumerate}
	\item [(B2.i)] $(t,\zeta,x)\mapsto a_{ij}(t,\zeta,x)$ and $(t,\zeta,x)\mapsto b_i(t,\zeta,x)$ are bounded on $(0,T)\times \mathcal{SP}\times \mathbb{R}^d$.
	\item[(B2.ii)] $x \mapsto a_{ij}(t,\zeta,x)$ and $x \mapsto b_i(t,\zeta,x)$ are continuous for each $\zeta \in \mathcal{SP}$ and $dt$-a.a. $t \in (0,T)$.
	\item[(B2.iii)] If $\zeta_n \longrightarrow \zeta$ weakly in $\mathcal{SP}$, then $a_{ij}(t,\zeta_n,x)\longrightarrow a_{ij}(t,\zeta,x)$ and $b_{i}(t,\zeta_n,x)\longrightarrow b_{i}(t,\zeta,x)$ locally uniformly in $x\in \mathbb{R}^d$ for each $t \in (0,T)$. 
\end{enumerate}
\begin{exa}
	\begin{enumerate}
		\item [(i)] Assumption B2 is fulfilled in the situation of \cite[Thm.5.3]{CG19} with $\sigma \equiv 0$. However, our Assumption B2 is considerably weaker than the assumptions in \cite[Thm.5.4]{CG19}, which implies uniqueness of solutions to \eqref{NLFPKE_Intro}-\eqref{Initial-cond_Intro}. Indeed, in contrast to \cite[Thm.5.4]{CG19}, we do neither assume Lipschitz continuity nor $C^m$-regularity for $m>0$ for $a$ and $b$.
		\item [(ii)] In \cite{Scheutzow87}, examples for nonlinear FPK equations with nonunique solutions are provided in the case $a\equiv 0$ and with dependence of $b$ on $\zeta$ of type $b(\zeta) = \int hd\zeta$, even for $h \in C_c(\mathbb{R}^d)$. Then, Assumption B2 is fulfilled. From here, generalizations to nonzero (degenerate) diffusive terms $a$ can be obtained.
	\end{enumerate}
\end{exa}
We still consider the choices \eqref{Choice-H-and-A_first-nonlin} and obtain:

\begin{prop}\label{1_prop_Application-CG19}
	Suppose $a$ and $b$ fulfill Assumption B2. Then there exists an entire subprobability solution flow to \eqref{NLFPKE_Intro}-\eqref{Initial-cond_Intro}. 
\end{prop}
\begin{proof}
	Comparing with the proof of Proposition \ref{1_prop_NL_entire-SP}, and since existence of solutions to the Cauchy problem follows, e.g., from \cite[Thm.5.3]{CG19}, it remains to prove closedness of $\mathcal{M}_{s,\nu}$ in $C([s,T],\mathcal{SP})$.
	Inspecting the proof of Proposition \ref{1_prop_NL_entire-SP}, (B1.iii) is only used in the case $\zeta_n = \mu^{(n)}_t$ and $\zeta = \mu_t$ such that $(\mu^{(n)}_t)_{t \in [s,T]} \in \mathcal{M}_{s,\nu}$ converges to $(\mu_t)_{t \in [s,T]}$ in $C([s,T],\mathcal{SP})$. In this situation, Lemma \ref{1_lem_app_precompact screenshots} implies tightness of $\{\mu^{(n)}_t\}_{n \in \mathbb{N}}$. Hence, the vague convergence $\mu^{(n)}_t \longrightarrow \mu_t$ is actually weak and (B2.iii) can be used to conclude the proof as done in the proof of Proposition \ref{1_prop_NL_entire-SP}.
\end{proof}
\begin{rem}\label{1_rem_sub_already_prob}
	Similarly as under Assumption A2 in the linear case, under Assumption B2 any solution with initial condition $\nu \in \mathcal{P}$ is a probability solution. Therefore, it is clear that in this case one can also select an entire probability flow (for example, by considering the restriction of the flow of Proposition \ref{1_prop_Application-CG19} to initial data $\nu \in \mathcal{P}$).
\end{rem}

\subsubsection{Nemytskii-type coefficients}
Here we let $\mathcal{S}_0 = \mathcal{SP}_{\ll}$ and consider, for $1 \leq i,j \leq d$, Borel coefficients 
\begin{equation*}
	a_{ij}: (0,T)\times \mathbb{R}^d \to \mathbb{R},\quad \tilde{b}_i: (0,T)\times \mathbb{R}\times \mathbb{R}^d \to \mathbb{R},
\end{equation*}
i.e. $a$ is independent of $\zeta \in \mathcal{SP}$ and $b$ acts on $(t,\zeta,x)$ via $b: (t,\zeta,x) \mapsto \tilde{b}(t,(d\zeta/dx)(x),x)$. To define $b$ pointwise, we agree to denote by $d\zeta /dx$ the version of the density of $\zeta$ with respect to $dx$ which is $0$ on the complement of its Lebesgue points. This choice renders $b$ $\mathcal{B}((0,T))\otimes \tau_v \otimes \mathcal{B}(\mathbb{R}^d)$-measurable. In this situation, solutions $(\mu_t)_{t \in (s,T)}$ to \eqref{NLFPKE_Intro}-\eqref{Initial-cond_Intro} are represented by densities $(\rho_t)_{t \in (s,T)}$, $\rho_t = d\mu_t /dx$, and \eqref{NLFPKE_Intro} turns into the nonlinear PDE
\begin{equation}\label{PDE-eq-Nemytskii}
	\partial_t \rho_t(\cdot) = \partial_{ij}\big((a_{ij}(t,\cdot)\rho_t(\cdot)\big)-\partial_i\big(b_i(t,\rho_t(\cdot),\cdot)\rho_t(\cdot)\big).
\end{equation}
We identify $(\mu_t)_{t \in (s,T)}$ and $(\rho_t)_{t \in (s,T)}$. We impose the following conditions on $a$ and $b$. 
\\
\\
\textbf{Assumption N1.} 
\begin{enumerate}
	\item [(N1.i)] $\sup_{t_1 \leq t \leq t_2}||a(t,\cdot)||_{W^{1,\infty}(B_R)} < \infty $ $\forall 0 < t_1 < t_2 < T, R>0$.
	\item[(N1.ii)] For any $0 < t_1 < t_2 < T, R>0$, there are constants $0 < \lambda_1 < \lambda_2$ depending on $t_1,t_2$ and $R$ such that $\lambda_1 |\xi|^2 \leq a_{ij}(t,x)\xi_i\xi_j \leq \lambda_2|\xi|^2$ for all $(t,x) \in [t_1,t_2]\times B_R$ and $\xi  = (\xi_1,\dots,\xi_d)\in \mathbb{R}^d$.
	\item[(N1.iii)] $r\mapsto \tilde{b}(t,r,x) \in C(\mathbb{R})$ $\forall (t,x) \in (0,T)\times \mathbb{R}^d$, and $||\tilde{b}||_{L^\infty(\mathcal{T}\times \mathbb{R}\times K)}<\infty$ $\forall \,\mathcal{T}\times K \subset \subset (0,T)\times\mathbb{R}^d$.
	\item[(N1.iv)] $\int_0^T \sup_{(r,x)\in \mathbb{R}\times B_R}\big(|a(t,x)|+|\tilde{b}(t,r,x)|\big)dt <\infty$ $\forall R>0.$
\end{enumerate}
Any solution $\mu= (\mu)_{t \in (s,T)} = (\rho_t)_{t \in (s,T)}$ to \eqref{NLFPKE_Intro}-\eqref{Initial-cond_Intro} is a solution to the \textit{linear} FPK equation with Borel coefficients $(t,x)\mapsto a(t,x)$ and $(t,x)\mapsto \tilde{b}(t,\rho_t(x),x)$. Due to (N1.i)-(N1.iii), \cite[Cor.6.4.3.]{FPKE-book15} applies to any such $(\mu_t)_{t \in (s,T)}$, any $p >d+2$ and $(t_1,t_2)\times B_R \subset \subset (0,T)\times \mathbb{R}^d$. Consequently, by \cite[Thm.6.2.2.]{FPKE-book15}, there is $\gamma>0$ and a $dxdt$-version of $\rho$ (for simplicity again denoted $\rho$) such that $\mu= \rho \in C^\gamma_{\loc}((s,T),C^\gamma_{\loc}(\mathbb{R}^d))$. $\gamma$ does not depend on $\mu$ or $\nu$. In particular, we have
$$\mathcal{M}_{s,\nu} \subseteq C^\gamma_{\loc}((s,T),C^\gamma_{\loc}(\mathbb{R}^d))$$
for each $\nu \in \mathcal{SP}$.
With regard to Theorem \ref{Thm-1}, we choose $(\mathcal{H},\tau)$ as the nonnegative functions with $L^1(dx)$-norm less or equal to $1$ in $C^0_{\loc}(\mathbb{R}^d)$ (where $C^0_{\loc}(\mathbb{R}^d)$ denotes the metrizable topological space of continuous functions on $\mathbb{R}^d$ with the topology of locally uniform convergence, which is stronger than the vague topology). The key to flow selections to equation \eqref{PDE-eq-Nemytskii} is the following lemma. Note that it is valid for any initial condition $\nu \in \mathcal{SP}$, not only for $\nu \in \mathcal{SP}_{\ll}$.
\begin{lem}\label{Prop-key-Nemytskii}
	Let $C_{s,T}C^0_{\loc}(\mathbb{R}^d)$ be equipped with the topology of locally uniform convergence on $(s,T)$ in $C^0_{\loc}(\mathbb{R}^d)$, and let $(s,\nu) \in [0,T)\times \mathcal{SP}$. Then
	\begin{enumerate}
		\item [(i)] $\mathcal{M}_{s,\nu}$ is relatively compact in $C_{s,T}C^0_{\loc}(\mathbb{R}^d)$.
		\item[(ii)] $\mathcal{M}_{s,\nu}$ is closed in $C_{s,T}C^0_{\loc}(\mathbb{R}^d)$.
	\end{enumerate}
\end{lem}
\begin{proof}
	\begin{enumerate}
		\item [(i)] By the compact embedding $C_{\loc}^\gamma((s,T),C^\gamma_{\loc}(\mathbb{R}^d)) \subseteq C_{s,T}C^0_{\loc}(\mathbb{R}^d)$, it suffices to prove
		\begin{equation}\label{Bound-Hldernorm}
			\sup_{\mu \in \mathcal{M}_{s,\nu}}||\mu||_{C^\gamma([t_1,t_2],C^\gamma(B_R))} < \infty \quad \forall s < t_1 < t_2 < T, R>0.
		\end{equation}
		To this end, we will use \cite[Thm.6.2.2.(ii)]{FPKE-book15} to obtain for any $p > d+2$
		\begin{equation}\label{one-eq}
			\sup_{\mu \in \mathcal{M}_{s,\nu}}||\mu||_{C^\gamma([t_1,t_2],C^\gamma(B_R))}\leq N_1(d,p,\gamma,t_1,t_2,R)||\mu||_{\mathcal{H}^{p,1}(B_R,(t_1,t_2))},
		\end{equation}
		where the finite constant $N_1$ depends only on the indicated quantities, i.e. in particular not on $\mu$. Moreover, for any $p > d+2$, by \cite[Cor.6.4.5.]{FPKE-book15}, and since each $\mu_t$ is a nonnegative subprobability measure, we have
		\begin{equation}\label{two-eq}
			\sup_{\mu \in \mathcal{M}_{s,\nu}}||\mu||_{\mathbb{H}^{p,1}(B_R,(t_1,t_2))} \leq N_2 < \infty,
		\end{equation}
		where $N_2$ depends on 
		$$d,p,t_1,t_2, \sup_{t \in [t_1,t_2]}||a(t,\cdot)||_{W^{1,\infty}(B_R)}, \sup_{(t,r,x) \in [t_1,t_2]\times\mathbb{R}\times B_R}|b(t,r,x)|$$
		and $\inf_{(t,x)\in [t_1,t_2]\times B_R}\det a(t,x)$ (which is strictly positive due to (N1.ii)), but not on $\mu$. Therefore, to obtain \eqref{Bound-Hldernorm}, it remains to prove
		\begin{equation}\label{eq-to-prove-H-mathcalH}
			||\mu||_{\mathcal{H}^{p,1}(B_R,(t_1,t_2))} \leq N_3 ||\mu||_{\mathbb{H}^{p,1}(B_R,(t_1,t_2))} \quad \forall s < t_1 < t_2 < T, R >0,
		\end{equation}
		for each $\mu \in \mathcal{M}_{s,\nu}$, for a finite constant $N_3$, which may depend on $t_1,t_2,R, p$, $d$ and the values of $a$ and $b$ considered above, but not on $\mu$. Recall the definition of the norm on $\mathcal{H}^{p,1}(B_R,(t_1,t_2))$, i.e.
		$$||\mu||_{\mathcal{H}^{p,1}(B_R,(t_1,t_2))} = ||\mu||_{\mathbb{H}^{p,1}(B_R,(t_1,t_2))}+||\partial_t\mu||_{\mathbb{H}^{p,-1}(B_R,(t_1,t_2))}.$$ Keeping $t_1,t_2,R$ fixed, we abbreviate the above norms in the previous line by $||\cdot||_{\mathcal{H}^{p,1}}$, $||\cdot||_{\mathbb{H}^{p,1}}$ and $||\cdot||_{\mathbb{H}^{p,-1}}$. We have
		\begin{align}\label{aux-eq_1}
			||\partial_t\mu||_{\mathbb{H}^{p,-1}} \leq ||\partial_t\mu - \partial_{ij}(a_{ij}\mu)||_{\mathbb{H}^{p,-1}} + ||\partial_{ij}(a_{ij}\mu)||_{\mathbb{H}^{p,-1}}.
		\end{align}
		Concerning the second summand, we find
		\begin{align*}
			||\partial_{ij}(a_{ij}\mu)||_{\mathbb{H}^{p,-1}} &= \sup_{h \in \mathbb{H}^{p',1}_0((B_R,(t_1,t_2))}-\int_{t_1}^{t_2}\int_{B_R}\partial_i(a_{ij}\rho) \partial_{j}h\,dxdt,
		\end{align*}
		and for each $h$ appearing in the above supremum, the integral term can be estimated as follows:
		\begin{align*}
			&-\int_{t_1}^{t_2}\int_{B_R}\partial_i(a_{ij}\rho)\partial_j h \,dx dt \leq
			\bigg(\int_{t_1}^{t_2}\int_{B_R}|\nabla h|^{p'}dxdt\bigg)^{1/{p'}}\\& \quad\quad\quad\quad\times \bigg[\bigg(\int_{t_1}^{t_2}\int_{B_R}|\rho \divv a|^pdxdt\bigg)^{1/p} + \bigg(\int_{t_1}^{t_2}\int_{B_R}|\nabla \rho\cdot a|^pdxdt\bigg)^{1/p}\bigg] \\& \leq
			2||h||_{\mathbb{H}^{p',1}}\cdot \sup_{t \in (t_1,t_2)}||a(t,\cdot)||_{W^{1,\infty}(B_R)}||\rho||_{\mathbb{H}^{p,1}}.
		\end{align*}
		This yields
		\begin{equation}\label{aux-eq_0}
			||\partial_{ij}(a_{ij}\mu)||_{\mathbb{H}^{p,-1}} \leq 2\sup_{t \in (t_1,t_2)}||a(t,\cdot)||_{W^{1,\infty}(B_R)}||\rho||_{\mathbb{H}^{p,1}}.
		\end{equation}
		We proceed with the first summand of \eqref{aux-eq_1}, noting that
		\begin{align*}
			||\partial_t\mu - \partial_{ij}(a_{ij}\mu)||_{\mathbb{H}^{p,-1}} = \sup_{h \in C^\infty_c((t_1,t_2)\times B_R)} \int_{t_1}^{t_2}\int_{B_R} - (\partial_th+a_{ij}\partial_{ij}h)\rho\, dx dt.
		\end{align*}
		For each $h \in C^\infty_c((t_1,t_2)\times B_R)$, the integral term on the right-hand side can further be treated as follows.
		\begin{align*}
			\int_{t_1}^{t_2}&\int_{B_R} - (\partial_th+a_{ij}\partial_{ij}h)\rho \,dx dt = \int_{t_1}^{t_2}\int_{B_R}\tilde{b}(t,\rho_t(x),x)\cdot \nabla h(t,\cdot)(x)\rho_t(x)dxdt \\&
			\leq \sup_{(t,r,x)\in (t_1,t_2)\times \mathbb{R}\times B_R}|\tilde{b}(t,r,x)|\bigg(\int_{t_1}^{t_2}||\rho_t(\cdot)||^p_{L^p(B_R)}dt\bigg)
			^{1/p}\bigg(\int_{t_1}^{t_2}||\nabla h(t,\cdot)||^{p'}_{L^{p'}(B_R)}dt\bigg)^{1/{p'}},
		\end{align*}
		where we have used that $(\mu_t)_{t \in (s,T)} = (\rho_t)_{t \in (s,T)}$ fulfills \eqref{Sol_time-space-formulation}. Hence, 
		\begin{equation}\label{aux-eq_2}
			||\partial_t\rho - \partial_{ij}(a_{ij}\rho)||_{\mathbb{H}^{p,-1}} \leq ||\tilde{b}||_{L^\infty((t_1,t_2)\times \mathbb{R}\times B_R)}||\rho||_{\mathbb{H}^{p,1}},
		\end{equation}
		and \eqref{aux-eq_1}-\eqref{aux-eq_2} yield 
		\begin{align*}
			||\mu||_{\mathcal{H}^{p,1}(B_R,(t_1,t_2))} \leq (1+2\sup_{t \in (t_1,t_2)}||a(t,\cdot)||_{W^{1,\infty}(B_R)}+||\tilde{b}||_{L^\infty((t_1,t_2)\times \mathbb{R}\times B_R)})||\mu||_{\mathbb{H}^{p,1}(B_R,(t_1,t_2))}.
		\end{align*}
		We have proved \eqref{eq-to-prove-H-mathcalH} with $N_3= 1+2\sup_{t \in (t_1,t_2)}||a(t,\cdot)||_{W^{1,\infty}(B_R)}+||\tilde{b}||_{L^\infty((t_1,t_2)\times \mathbb{R}\times B_R)}$, which is finite due to Assumption N1, and independent of $\mu$. Now \eqref{one-eq}-\eqref{eq-to-prove-H-mathcalH} imply \eqref{Bound-Hldernorm}, which completes the proof of (i).
		
		\item[(ii)] Let $\{\mu^{(n)}\}_{n \in \mathbb{N}}\subseteq \mathcal{M}_{s,\nu}$ with $\mu^{(n)} = (\mu^{(n)}_t)_{t \in (s,T)} = (\rho^{(n)}_t)_{t \in (s,T)}$ be convergent in $C_{s,T}C^0_{\loc}(\mathbb{R}^d)$ with limit $\rho$. Clearly, $\rho$ is nonnegative and $\rho(t,\cdot)$ is the density of a Borel subprobability measure $\mu_t$ for each $t$. By the continuity of $\rho$, $(t,x)\mapsto \tilde{b}(t,\rho_t(x),x)$ is Borel and $t \mapsto \mu_t$ is vaguely continuous. Assumption N1 entails (i) of Definition \ref{Def_NLFPKE-sol}. Concerning \eqref{Sol-eq_NLFPKE_cont}, for each $\varphi \in C^2_c(\mathbb{R}^d)$ and $(t,x)\in (s,T)\times\mathbb{R}^d$, we have
		$$\tilde{b}(t,\rho_t^{(n)}(x),x)\cdot \nabla\varphi(x)\longrightarrow \tilde{b}(t,\rho_t(x),x)\cdot \nabla \varphi(x),$$
		and, moreover, $\rho_t^{(n)}(\cdot) \longrightarrow \rho_t(\cdot)$ uniformly on $\supp \varphi$. Hence, by (N1.iii) and since $a$ does not depend on $\rho^{(n)}$,
		$$\int_{\mathbb{R}^d} \mathcal{L}_{t,\mu_t^{(n)}}\varphi(x)\rho_t^{(n)}(x)dx \longrightarrow \int_{\mathbb{R}^d} \mathcal{L}_{t,\mu_t}\varphi(x)\rho_t(x)dx.$$ By Lebesgue's dominated convergence and (N1.iv), we obtain 
		$$\int_s^T\int_{\mathbb{R}^d}\mathcal{L}_{t,\mu_t^{(n)}}\varphi(x)\rho_t^{(n)}(x)dxdt \longrightarrow \int_s^T\int_{\mathbb{R}^d}\mathcal{L}_{t,\mu_t}\varphi(x)\rho_t(x)dxdt.$$
		This proves that $(\mu_t)_{t \in (s,T)}$ belongs to $\mathcal{M}_{s,\nu}$. 
	\end{enumerate}
\end{proof}

The following corollary follows immediately from the previous lemma and Theorem \ref{Thm-1}.
\begin{kor}\label{Cor_Nemytskii-case}
	Let $\mathcal{C}_s \subseteq \mathcal{SP}$ for each $s \in [0,T)$ such that the family 
	\begin{equation*}
		\mathcal{A}_{s,\nu} := \begin{cases}
			\mathcal{M}_{s,\nu},& \text{ if }\nu \in \mathcal{C}_{s}\\
			\emptyset,& \text{ else }
		\end{cases}
	\end{equation*}
	is flow-admissible. Then, if the coefficients fulfill Assumption N1, there exists a flow to \eqref{PDE-eq-Nemytskii} with respect to $\{\mathcal{A}_{s,\nu}\}_{(s,\nu)\in [0,T)\times \mathcal{SP}}$.
\end{kor}

We close this section with several applications of the previous result. In comparison with the papers cited below, here we consider only the drift $b$ of Nemytskii-type, while the diffusion $a$ only depends on $(t,x) \in (0,T)\times \mathbb{R}^d$.

\paragraph{Flow selection under assumptions of \cite{BR18_2}}
Dropping the nonlinearity of $a$, in \cite{BR18_2} solutions to \eqref{PDE-eq-Nemytskii} are obtained for the case of time-independent coefficients under the following assumptions: 
$$a_{ij}\in C^2(\mathbb{R}^d)\cap C^1_b(\mathbb{R}^d), \,\sum_{i,j \leq d}a_{ij}\xi_i\xi_j \geq \gamma |\xi|^2 \,\,\forall \xi, x \in \mathbb{R}^d\text{ for }\gamma >0,$$
$$b \in C^1(\mathbb{R}\times \mathbb{R}^d)\cap L^\infty(\mathbb{R}\times \mathbb{R}^d).$$ Moreover, $b(0,x) = 0$ $\forall x \in \mathbb{R}^d$.  Obviously, these assumptions are stronger than Assumption N1. By \cite[Thm.3.4.]{BR18_2}, there is a solution to \eqref{PDE-eq-Nemytskii} for each initial condition $(s,\nu) \in [0,T)\times L^1(\mathbb{R}^d)$ and this solution family is a semigroup in $L^1(\mathbb{R}^d)$. In particular, since the coefficients here are time-independent, this family has the flow property, and its restriction to initial data from $\mathcal{SP}_{\ll}$ or $\mathcal{P}_{\ll}$is an entire (sub-) probability flow.

Without relying on the specific construction of solutions considered in \cite{BR18_2}, but only by the sheer existence result of \cite[Thm.3.4.]{BR18_2} restricted to initial data in $\mathcal{SP}_{\ll}$, we can apply Corollary \ref{Cor_Nemytskii-case} with $\mathcal{C}_s = \mathcal{SP}_{\ll}$ to also obtain the existence of a flow in $\mathcal{SP}_{\ll}$. Since any solution curve has constant total mass due to the boundedness of $a$ and $b$, its restriction to initial data in $\mathcal{P}_{\ll}$ is an entire probability flow.
It is not clear to us whether the semigroup constructed in \cite{BR18_2} is related to or coincides with the flow obtained by our techniques (of course they coincide in the case of well-posedness).

\paragraph{Flow selection for time-dependent coefficients as in \cite{BR21_NLFP-time-dep}}
In \cite{BR21_NLFP-time-dep}, time-dependent coefficients are considered. It is obvious that assumptions (H1)-(H2) of \cite{BR21_NLFP-time-dep} imply Assumption N1 (assuming $a$ depends only on $(t,x)$). By \cite[Rem.2.2]{BR21_NLFP-time-dep}, for any $(s,\nu) \in [0,T)\times \mathcal{SP}_{\ll}$, there is a solution to \eqref{PDE-eq-Nemytskii} with initial data $(s,\nu)$ (attained in the sense of weak convergence of measures). Hence, as in the previous paragraph, choosing $\mathcal{C}_s$ as $\mathcal{SP}_{\ll}$, Corollary \ref{Cor_Nemytskii-case} and Theorem \ref{Thm-1} yield flows for \eqref{PDE-eq-Nemytskii} in $\mathcal{SP}_{\ll}$ and, by restriction of such a flow to initial data being probability measures, in $\mathcal{P}_{\ll}$.

\paragraph{General measures as initial data (\cite{NLFPK-DDSDE5})}
Choosing $a = C\Id$ and $b(t,r,x) = D(x)\tilde{b}(r)$, we are in the setting of \cite{NLFPK-DDSDE5} (with $\beta(r) = Cr$). Clearly, assumptions (k)-(kkk) and (5.3) of \cite[Sect.5]{NLFPK-DDSDE5} imply Assumption N1. By Theorem 5.2., in this situation \eqref{PDE-eq-Nemytskii} has a solution for each finite Radon measure as initial datum. These solutions belong to $L^\infty((0,T),L^1)\cup \bigcap_{\delta >0}L^\infty((\delta,T)\times \mathbb{R}^d)$, preserve positivity and total mass, but it is not clear whether they are unique, see Remark 5.3. We apply Corollary \ref{Cor_Nemytskii-case} with $\mathcal{C}_s = \mathcal{SP}$ and obtain via Theorem \ref{Thm-1} an entire subprobability flow.  Its restriction to initial data in $\mathcal{P}$ is an entire probability flow, due to the boundedness of $a$ and $b$. In Remark 5.3., the authors of \cite{NLFPK-DDSDE5} state that it is not clear whether the family of their constructed solutions for every finite Radon measure $\nu$ has the flow property. With our methods, we obtain the existence of such a flow, at least in the class of nonnegative finite Radon measures. Indeed, since in the present case solutions preserve the total mass of their initial datum, it suffices to construct one flow in each of the disjoint "channels" $c\mathcal{P} = \{\eta: \eta = c\mu, \mu \in \mathcal{P}\}$, $c \geq 0$.

\section{Comparison to Markovian semigroups}\label{Section-Markov_semigroups}
Here we study cases of coefficients $a$ and $b$ independent of $\zeta \in \mathcal{SP}$ where the flows constructed in the proof of Theorem \ref{Thm-1} fulfill the Chapman-Kolmogorov equations \eqref{CK-eq}, suggesting that our selection on the level of the FPK equation seems to be the right analogue to the Markovian selections for martingale problems in \cite[Ch.12]{StroockVaradh2007}.

We consider bounded Borel coefficients $a$ and $b$ on $(0,T)\times \mathbb{R}^d$ such that $a$ and $b$ are continuous in $x$ for fixed $t$. In particular, by Remark \ref{1_rem_aux1} and Proposition \ref{Prop_entire_prob}, there is an entire probability flow to \eqref{NLFPKE_Intro}-\eqref{Initial-cond_Intro}. Note that the current assumptions coincide with those in \cite[Ch.12]{StroockVaradh2007} and \cite{Rehmeier_Flow-JEE}.

Subsequently we allow the initial time $s=T$ and set $\mathcal{M}^1_{T,\nu} = \{\nu\}$. If $\delta_x$ appears as an index, we simply write $x$, if no confusion will occur.

\begin{dfn}\label{Markov semigr}
	A family of solutions $\{\mu^{s,x}\}_{(s,x) \in [0,T]\times \mathbb{R}^d}$ with $\mu^{s,x} \in \mathcal{M}^1_{s,x}$ to \eqref{NLFPKE_Intro}-\eqref{Initial-cond_Intro} is a \textup{Markovian semigroup}, if for any $0 \leq s \leq t \leq T$, $y \mapsto \mu^{s,y}_t$ is $\mathcal{B}(\mathbb{R}^d)/\mathcal{B}(\mathcal{P})$-measurable and \eqref{CK-eq} holds. 
	
	For a Markovian semigroup, we define $\mu^{s,\nu} :=  \int_{\mathbb{R}^d} \mu^{s,y}d \nu(y)$ for any non-Dirac initial condition $\nu \in \mathcal{P}$ and call $\{\mu^{s,\nu}\}_{(s,\nu) \in [0,T]\times \mathcal{P}}$ the \textup{convex extension} of $\{\mu^{s,x}\}_{(s,x) \in [0,T]\times \mathbb{R}^d}$.
\end{dfn}
The following remark (which holds true also without the present continuity- and boundedness-assumptions on $a$ and $b$) shows that flows to \eqref{NLFPKE_Intro}-\eqref{Initial-cond_Intro} are more general than Markovian semigroups.
\begin{rem}\label{Rem_stable-convex}
	If $y \mapsto \mu^{s,y}$ is measurable for each $s$, where $\mu^{s,y}$ belongs to $\mathcal{M}^1_{s,\nu}$, then it is easy to see that the curve $t \mapsto \int_{\mathbb{R}^d} \mu^{s,y}_td\nu(y)$ belongs to $\mathcal{M}^{1}_{s,\nu}$ for any $\nu \in \mathcal{P}$. Moreover, the convex extension of a Markovian semigroup $\{\mu^{s,x}\}_{(s,x) \in [0,T]\times \mathbb{R}^d}$ is an entire probability flow, since we have
	\begin{equation*}
		\mu^{s,\nu}_t = \int_{\mathbb{R}^d} \mu^{s,y}_td\nu(y) = \int_{\mathbb{R}^d} \bigg(\int_{\mathbb{R}^d} \mu^{r,z}_t d\mu^{s,y}_r(z)\bigg)d \nu(y)=\int_{\mathbb{R}^d} \mu^{r,z}_t d\mu^{s,\nu}_r(z) = \mu^{r,\mu^{s,\nu}_r}_t.
	\end{equation*}
\end{rem}
The previous observation does not hold for coefficients depending on $\zeta \in \mathcal{P}$, since in that case solutions are not stable under convex combinations.

Next, we show that under the present assumptions on $a$ and $b$, both concepts coincide, at least for flows selected as in the proof of Theorem \ref{Thm-1}.

\begin{prop}\label{Prop Markov sg}
	If $a$ and $b$ are bounded and continuous in $x \in \mathbb{R}^d$, any entire probability flow $\{\mu^{s,\nu}\}_{(s,\nu) \in [0,T]\times \mathcal{P}}$ selected as in the proof of Theorem \ref{Thm-1} is the convex extension of a Markovian semigroup. In particular, the family $\{\mu^{s,x}\}_{(s,x) \in [0,T]\times \mathbb{R}^d}$ fulfills \eqref{CK-eq}.
\end{prop}

\begin{proof}
	Consider an entire probability flow $\{\mu^{s,\nu}\}_{(s,\nu) \in [0,T]\times \mathcal{P}}$ selected as in the proof of Theorem \ref{Thm-1}, with the topological choices made in Subsection \ref{Subsubsection-prob-flow-lin} (we set $\mu^{T,\nu} := \nu$). Since the measurability of $y \mapsto \mu_t^{s,y}$ follows from Lemma \ref{Lem Dirac-select is mable}, it remains to prove $\mu^{s,\nu}_t = \eta^{s,\nu}_t := \int \mu^{s,y}_t d\nu(y)$ for all $0 \leq s \leq t \leq T$ and $\nu \in \mathcal{P}$, since in this case, for all $0 \leq s \leq r \leq t \leq T$ and $\nu \in \mathcal{P}$, we have
	$$\mu^{s,x}_t = \mu^{r,\mu^{s,x}_r}_t = \int_{\mathbb{R}^d} \mu^{r,y}_t d\mu^{s,x}_r(y).$$
	By Remark \ref{Rem_stable-convex}  $\eta^{s,\nu}$ belongs to $\mathcal{M}^1_{s,\nu}$. Inspecting the proof of Theorem \ref{Thm-1} and using its notation, we have 
	\begin{equation}\label{geq1}
		\int_{\mathbb{R}^d} h_{n_{m^s_0}}d\mu^{s,\nu}_{q_{m^s_0}} \geq \int_{\mathbb{R}^d} h_{n_{m^s_0}}d\eta^{s,\nu}_{q_{m^s_0}}.
	\end{equation} Since Lemma \ref{Disint for any FPKE-sol} yields $\mu^{s,\nu} = \int \mu^y d\nu(y)$ for a family $\{\mu^y\}_{y \in \mathbb{R}^d}$ with $\mu^y \in \mathcal{M}^1_{s,y}$ $\nu$-a.s. such that $y \mapsto \mu^y_t$ is measurable for each $t\in [s,T]$ (we suppress the dependence of $\mu^y$ on $s$ in the notation), and since the selection method of each $\mu^{s,y}$ in $\mathcal{M}^1_{s,y}$ implies $\nu$-a.s.
	\begin{equation}\label{ineq y}
		\int_{\mathbb{R}^d} h_{n_{m^s_0}}d\mu^{s,y}_{q_{m^s_0}} \geq \int_{\mathbb{R}^d} h_{n_{m^s_0}}d\mu^{y}_{q_{m^s_0}},
	\end{equation}
	we deduce
	\begin{align}\label{leq1}
		\notag \int_{\mathbb{R}^d} h_{n_{m^s_0}}d\eta^{s,\nu}_{q_{m^s_0}} = \int_{\mathbb{R}^d}&\bigg( \int_{\mathbb{R}^d} h_{n_{m^s_0}}d\mu^{s,y}_{q_{m^s_0}}\bigg)d\nu(y)\\& \geq \int_{\mathbb{R}^d}\bigg( \int_{\mathbb{R}^d} h_{n_{m^s_0}}d\mu^{y}_{q_{m^s_0}}\bigg)d\nu(y) =  \int_{\mathbb{R}^d} h_{n_{m^s_0}}d\mu^{s,\nu}_{q_{m^s_0}}.
	\end{align}Hence, \eqref{geq1} is an equality. In turn, also \eqref{leq1} is a chain of equalities, which particularly gives that \eqref{ineq y} is an equality $\nu$-a.s.
	By iteration, \eqref{geq1} and \eqref{leq1} are equalities and \eqref{ineq y} is an equality $\nu$-a.s. for any index pair $(n_{m^s_k},q_{m^s_k})$, $k \in \mathbb{N}_0$. By continuity, we conclude $\mu^{s,\nu} = \eta^{s,\nu}$, which implies the claim.
\end{proof}
We conclude this section with a series of lemmas used for the previous proof.  We denote by $\pi_t$ the canonical projection $\pi_t : C([s,T],\mathbb{R}^d) \to \mathbb{R}^d$, $\pi_t(f) = f(t)$. The framework for the next results is contained in Appendix \ref{App_meas-select}.
\begin{lem}\label{1_lem_aux_meas1}
	The map $\mathcal{P}\ni \nu \mapsto \mathcal{M}_{s,\nu}$ is $\mathcal{B}(\mathcal{P})/\mathcal{B}(\comp(C([s,T],\mathcal{SP}))$-measurable.
\end{lem}
\begin{proof}
	We only need to verify that Lemma \ref{app_gen_meas_selec_lem1} applies. But from Remark \ref{Rem_union-precompt_as-well} and Subsection \ref{Subsubsection-prob-flow-lin} we know that under the present assumptions on $a$ and $b$, $\cup_{\nu \in \mathcal{P}}\mathcal{M}^1_{s,\nu}$ is relatively compact and that \eqref{NLFPKE_Intro} is stable with respect to limits in $C([s,T],\mathcal{SP})$. Also, a solution to \eqref{NLFPKE_Intro}-\eqref{Initial-cond_Intro} which is the limit of a sequence of solutions with initial conditions $(\nu_n)_{n \in \mathbb{N}}$ has initial condition $\nu$, if $\nu_n \to \nu$ in $\mathcal{P}$. Hence Lemma \ref{app_gen_meas_selec_lem1} applies and gives the claim.
\end{proof}
\begin{lem}\label{Lem Dirac-select is mable}
	For each $(s,y) \in [0,T]\times \mathbb{R}^d$, let $\mu^{s,y} \in \mathcal{M}^1_{s,y}$ denote the solution selected by the iterative selection method in the proof of Theorem \ref{Thm-1}. 
	The mapping $y \mapsto \mu^{s,y}$ is $\mathcal{B}(\mathbb{R}^d)/\mathcal{B}(C([s,T],\mathcal{P}))$-measurable. In particular, for each $t \in [s,T]$, $y \mapsto \mu^{s,y}_t$ is $\mathcal{B}(\mathbb{R}^d)/\mathcal{B}(\mathcal{P})$-measurable.
\end{lem}
For the proof, we use the notation of the proof of Theorem \ref{Thm-1}.
\begin{proof}
	It suffices to prove Borel measurability of 
	\begin{equation}\label{Eq for mability}
		y \mapsto  \bigcap_{k \geq 0} M_k^{s,y}
	\end{equation} from $\mathbb{R}^d$ to $\comp(C([s,T]\mathcal{SP}))$, since then Lemma \ref{app_gen_meas_selec_lem2} implies the measurability of $y \mapsto \mu^{s,y} \in C([s,T],\mathcal{SP})$, which gives the claim, since $\mu^{s,y} \in C([s,T],\mathcal{P})$ and $\mathcal{B}(\mathcal{SP})$ restricted to $\mathcal{P}$ coincides with the Borel $\sigma$-algebra of $\tau_w$ on $\mathcal{P}$. Since the mapping (\ref{Eq for mability}) can be rewritten as $y \mapsto \underset{N \to \infty}{\text{lim}}X_N(y)$ (the limit is taken in $\comp(C([s,T],\mathcal{SP}))$, see Appendix \ref{App_meas-select}), where
	$$X_N: y \mapsto \bigcap_{0\leq  k \leq N} M_{k}^{s,y} = M_{N}^{s,y} \in \comp(C([s,T],\mathcal{SP})),$$
	it suffices to prove measurability of each $X_N$. Since $y \mapsto \delta_y$ is continuous from $\mathbb{R}^d$ to $(\mathcal{P},\tau_w)$, by definition of $M_{0}^{s,y}$ and Lemmas \ref{1_lem_aux_meas1} and \ref{app_gen_meas_selec_lem0}, $X_0$ is measurable. Iteratively applying Lemma \ref{app_gen_meas_selec_lem0} to the maps $M_{k}^{s,y} \mapsto M_{k+1}^{s,y}$ and using the continuity of each map $G^{s,y}_{k}$, the measurability of $X_N$ follows. The final assertion follows from the measurability of the maps $\pi_t$.
\end{proof}

The following lemma is somewhat converse to the first part of Remark \ref{Rem_stable-convex}.
\begin{lem}\label{Disint for any FPKE-sol}
	Let $(s,\nu) \in [0,T]\times \mathcal{P}$. For $\mu \in \mathcal{M}^1_{s,\nu}$, there is a family $\{\mu^y\}_{y \in \mathbb{R}^d}$ with $\mu^y \in \mathcal{M}^1_{s,y}$ $\nu$-a.s. such that $y \mapsto \mu^y_t$ is $\mathcal{B}(\mathbb{R}^d)/\mathcal{B}(\mathcal{P})$-measurable and $\mu_t = \int_{\mathbb{R}^d}\mu_t^yd\nu(y)$ for each $t\in [s,T]$. 
\end{lem}
\begin{proof}
	By the superposition principle (see \cite{Trevisan16}), for $\mu \in \mathcal{M}^1_{s,\nu}$ there is a probability measure $P = P_\mu \in \mathcal{P}(C([s,T],\mathbb{R}^d))$, which solves the martingale problem with coefficients $a,b$ such that $P \circ \pi_t^{-1} = \mu_t$ for each $t \in [s,T]$. Disintegrating $P$ with respect to $P \circ \pi_s^{-1} = \nu$, we obtain a $\nu-$a.s. unique family $\{P_y\}_{y \in \mathbb{R}^d}$ in $\mathcal{P}(C([s,T],\mathbb{R}^d))$ such that $y \mapsto P_y$ is measurable and $P = \int_{\mathbb{R}^d}P_y d\nu(y)$. Furthermore, $P_y$ is a solution to the associated martingale problem with initial condition $(s,y)$ for $\nu$-a.e. $y \in \mathbb{R}^d$, see \cite[Prop.2.8]{Trevisan16}. Since the curve of one-dimensional time marginals of a solution to the martingale problem with initial condition $(s,\nu)$ is a weakly continuous probability solution to the corresponding FPK equation with initial condition $(s,\nu)$, we have $(P_y \circ \pi_t^{-1})_{t \in [s,T]} \in \mathcal{M}^1_{s,y}$ for $\nu-$a.e. $y \in \mathbb{R}^d$ and $\mu_t = \int_{\mathbb{R}^d}P_y\circ \pi_t^{-1}d\nu(y)$, i.e. the claim follows with $\mu_t^y:= P_y \circ \pi_t^{-1}$.
\end{proof}


\appendix
\section{Auxiliary results and proofs}
\textit{Proof of Lemma \eqref{Lem_cont-wlog}.}
We abbreviate $\mathcal{L}_{t,\tilde{\mu}_t}$ by $\mathcal{L}_t$. Choosing $\varphi(t,x) = f(t)\phi(x)$ with $f \in C_c^{\infty}((s,T)), \phi \in C_c^{\infty}(\mathbb{R}^d)$, \eqref{Sol_time-space-formulation} gives
\begin{equation*}
	\int_{s}^{T}f'(t)\bigg(\int_{\mathbb{R}^d}\phi(x)d\mu_t(x)\bigg) dt = -\int_{s}^{T}f(t)\bigg(\int_{\mathbb{R}^d}\mathcal{L}_t\phi(x)d\mu_t(x)\bigg)dt.
\end{equation*}Therefore, and since $t\mapsto \int\mathcal{L}_t\phi \,d\mu_t \in L^1((s,T);dt)$ by assumption, the map $t \mapsto \int\phi \,d\mu_t$ belongs to the Sobolev space $W^{1,1}((s,T))$ with weak derivative $t \mapsto \int\mathcal{L}_t\phi \,d\mu_t$ $dt$-a.s. Hence, choosing a countable set $\mathcal{F} \subseteq C_c^{\infty}(\mathbb{R}^d)$ dense in $C_c(\mathbb{R}^d)$, there exists a real-valued map $(\phi,t) \mapsto F(\phi,t)$ on $\mathcal{F} \times [s,T]$ such that for each $\phi \in \mathcal{F}$, $t \mapsto F(\phi,t)$ is an absolutely continuous version of $t \mapsto \int \phi\, d\mu_t$. Let $\mathcal{T}$ denote the set of all $t \in [s,T]$ such that $\mu_t \in \mathcal{SP}$ and $F(\phi,t) = \int\phi\, d\mu_t$ for all $\phi \in \mathcal{F}$. By assumption, $\mathcal{T}^c$ has $dt$-measure $0$. If $t \in \mathcal{T}$, then
\begin{equation*}
	|F(\phi,t)-F(\phi',t)| \leq \int_{\mathbb{R}^d}|\phi-\phi'|d\mu_t \leq ||\phi-\phi'||_{\infty},
\end{equation*}whereby $\phi \mapsto F(\phi,t)$ is uniformly continuous on $\mathcal{F}$ and hence uniquely extends to a continuous linear map on all of $C_c(\mathbb{R}^d)$ (again denoted $F(\cdot,t)$).
Here and for the rest of the proof, $(\phi_n)_{n \geq 1} \subseteq \mathcal{F}$ denotes any sequence converging to $\phi$ in $C_c(\mathbb{R}^d)$. Thus, $F(\cdot,t) = \mu_t$ for each $ t \in \mathcal{T}$ as elements in the dual space of $C_c(\mathbb{R}^d)$. For $t \in \mathcal{T}^c$, we have for $\phi \in \mathcal{F}$
\begin{equation*}
	F(\phi,t) = \underset{n \to \infty}{\text{lim}}F(\phi,t_n) = \underset{n \to \infty}{\text{lim}}\int_{\mathbb{R}^d} \phi\, d\mu_{t_n}.
\end{equation*}Here and below, $(t_n)_{n\geq 1} \subseteq (s,T)\cap \mathcal{T}$ is any sequence converging to $t$. In particular, $F(\cdot,t)$ is linear and uniformly continuous on $\mathcal{F}$. For $\phi \in C_c(\mathbb{R}^d)\backslash \mathcal{F}$, we set $F(\phi,t) := \underset{l \to \infty}{\text{lim}}F(\phi_l,t)$ (with $(\phi_l)_{l \geq 1}$ as above), which is well-defined due to the uniform continuity of $\phi \mapsto F(\phi,t)$ on $\mathcal{F}$ and is hence a linear, positive functional on $C_c(\mathbb{R}^d)$ with $||F(\cdot,t)||_{C^*_c(\mathbb{R}^d)} \leq 1$. Therefore, the Riesz-Markov-Kakutani representation theorem implies the existence of a unique element $\mu'_t \in \mathcal{SP}$ such that $F(\cdot,t) = \mu'_t$. For $t \in [s,T]$, define
\begin{equation*}
	\bar{\mu}_t := \begin{cases}
		\mu_t,& t \in \mathcal{T} \\ \mu'_t,& t \in [s,T]\backslash\mathcal{T}.
	\end{cases}
\end{equation*}By definition, $[s,T] \ni t \mapsto \int \phi \,d\bar{\mu}_t$ is continuous for each $\phi \in \mathcal{F}$. Since for each $\phi \in C_c(\mathbb{R}^d)\backslash \mathcal{F}$, we have
\begin{equation*}
	\int \phi \,d\bar{\mu}_t = \underset{l \to \infty}{\text{lim}}\int \phi_l d\tilde{\mu}_t = \underset{l \to \infty}{\text{lim}}\underset{n \to \infty}{\text{lim}}\int \phi_l d\bar{\mu}_{t_n} = \underset{n \to \infty}{\text{lim}}\underset{l \to \infty}{\text{lim}}\int \phi_l d\bar{\mu}_{t_n} = \underset{n \to \infty}{\text{lim}}\int \phi d\bar{\mu}_{t_n},
\end{equation*}it follows that $t \mapsto \bar{\mu}_t$ is vaguely continuous and, by construction, coincides with $(\mu_t)_{t \in (s,T)}$ $dt$-a.s. It is clear that $\bar{\mu}$ satisfies \eqref{Sol_time-space-formulation}-\eqref{Sol_time-space-formulation_IC}. This and the continuity of $\bar{\mu}$ implies all assertions.
\qed
\\

In the remainder of this appendix, we prove Lemma \ref{1_lem_app_precompact screenshots}. We start with the following auxiliary observation.
\begin{lem}\label{1_lem_Ex of Lyapunov fct.}
	For any $\nu \in \mathcal{SP}$, there exists a nonnegative compact function $V=V_{\nu} \in C^2(\mathbb{R}^d)$ such that $\max_{1\leq i,j \leq d}(||\partial_iV||_{\infty},||\partial_{ij}V||_{\infty}) < \infty$ and $\int_{\mathbb{R}^d} Vd\nu < \infty$.
\end{lem}
\begin{proof}
	Since every single Borel probability measure on $\mathbb{R}^d$ is tight, there exists a sequence of strictly increasing radii $R_n>0$ such that $\nu\big(\overline{B_{R_n}}^c\big) \leq n^{-3}$. Without loss of generality, we may assume $R_{n+1} \geq R_n+1$. Set $W_{\nu}(x) := 1$ on $\overline{B_{R_1}}$ and $W_{\nu}(x) := n$ on $\overline{B_{R_{n+1}}}\backslash \overline{B_{R_n}}$. Clearly, $W_{\nu}$ is nonnegative, radial and has compact sublevel sets $\{W_{\nu}\leq c \} = \overline{B_{R_{\lfloor c \rfloor+1}}}$ for $c \geq 1$ and $\{W_{\nu} \leq c \} = \emptyset$ for $c < 1$. Furthermore, note that
	$$\int_{\mathbb{R}^d} W_{\nu}d\nu \leq \sum_{n \geq 1} n^{-2} < \infty.$$
	Now consider, for each $n \geq 1$, a function $h_n$, defined via 
	$$h_n(r) :=
	\begin{cases}
		n&, r \in [R_{n+1},R_{n+1}+\frac{1}{4}],\\
		g_n(r)&, r \in (R_{n+1}+\frac{1}{4},R_{n+1}+\frac{3}{4}),\\
		n+1&, r \in [R_{n+1}+\frac{3}{4},R_{n+2}],
	\end{cases}$$for a suitable increasing $C^2$-function $g_n:\mathbb{R}\to \mathbb{R}$ with bounded first- and second-order derivatives such that $h_n \in C^2([R_n,R_{n+1}])$. Note that we used the assumption $R_{n+1} \geq R_n +1$ for the definition of $h_n$. Clearly, the family $\{g_n\}_{n \geq 1}$ can be chosen with uniformly (in $n$) bounded first- and second-order derivatives. Compounding these functions, we note that $V_{\nu}: \mathbb{R}^d \to \mathbb{R}$, $$
	V_{\nu}(x) := \begin{cases}
		1&, x \in \overline{B_{R_2}}, \\
		h_n(|x|)&, x \in \overline{B_{R_{n+2}}}\backslash \overline{B_{R_{n+1}}},n \geq 1 ,
	\end{cases}$$
	is a nonnegative function in $C^2(\mathbb{R}^d)$ such that $V_{\nu}(x) \longrightarrow \infty$ as $|x| \to \infty$ with uniformly bounded first- and second-order partial derivatives and compact sublevel sets, i.e. it is a compact function as in the assertion. Finally, since $V_{\nu} \leq W_{\nu}$ by construction, $\int V_{\nu} d\nu <  \infty$ follows, which completes the proof.
\end{proof}
Now we obtain the following tightness result for solutions.
\begin{lem}\label{1_lem_app_precompact screenshots}
	Consider $\mathcal{B}((0,T))\otimes \tau_v\otimes \mathcal{B}(\mathbb{R}^d)$-measurable coefficients $a_{ij}, b_i: (0,T)\times \mathcal{SP}\times \mathbb{R}^d \to \mathbb{R}$ such that (B2.i) and (B2.ii) of Assumption B2 hold. Then, for each $0 \leq s \leq t \leq T$ and $\nu \in \mathcal{SP}$, $\{\mu_t: \mu \in \mathcal{M}_{s,\nu}\}$ is tight.
\end{lem}
\begin{proof}
	Fix $(s,\nu) \in [0,T] \times \mathcal{SP}$ and $t \in [s,T]$. Consider a function $V=V_{\nu} \in C^2(\mathbb{R}^d)$ with the properties stated in Lemma \ref{1_lem_Ex of Lyapunov fct.} and let $\{\varphi_l\}_{l \geq 1} \subseteq C^2_c(\mathbb{R}^d)$ have the following properties: $\varphi_l$ is nonnegative, increases pointwise to $V$ as $l \to \infty$ such that $\varphi_l = V$ on $B_l$ and such that $\partial_i \varphi_l, \partial_{ij} \varphi_l$ are bounded uniformly in $1\leq i,j\leq d$ and  $l>1$ by some number $0 <D < \infty$. Then, for any $(\mu_t)_{t \in [s,T]} \in \mathcal{M}_{s,\nu}$, Assumption (B2.i) entails
	\begin{equation*}
		\sup_{l \geq1} \bigg|\int_{s}^{t}\int_{\mathbb{R}^d}a_{ij}(r,\mu_r,x)\partial_{ij}\varphi_l(x)+b_i(r,\mu_r,x)\partial_i\varphi_l(x)d\mu_r(x)dr\bigg| < C,
	\end{equation*}
	with $C=C(D)>0$ independent of the particular solution $(\mu_t)_{t \in [s,T]} \in \mathcal{M}_{s,\nu}$. Hence, by \eqref{Sol-eq_NLFPKE_cont}, we have 
	$$\underset{l \geq 1}{\text{sup}}\bigg|	\int_{\mathbb{R}^d}\varphi_ld\mu_{t}- \int_{\mathbb{R}^d}\varphi_ld\nu \bigg|   < C,$$which, together with $\sup_{l \geq 1}\int \varphi_l d\nu = \int V_{\nu} d\nu < \infty$, entails a uniform in $\mathcal{M}_{s,\nu}$ bound on $\int V_{\nu} d\mu_t = \sup_{l \geq 1}\int \varphi_l \,d\mu_t$. Therefore, $\{\mu_t: \mu \in \mathcal{M}_{s,\nu}\}$ is tight.
\end{proof}

\section{Measurable selections}\label{App_meas-select}
For the convenience of the reader, we recall basic tools for measurable selections as needed in Section \ref{Section-Markov_semigroups}. This material is taken from \cite[Sect.12.1]{StroockVaradh2007}.

Let $(X,d)$ be a separable metric space and denote by $\comp(X)$ the set of nonempty compact subsets of $X$. For $K \in \comp(X)$ and $\varepsilon>0$ let $K_\varepsilon := \{x \in X: \dist(K,x) < \varepsilon\}$. The \textit{Hausdorff distance} $d_H$,
\begin{equation*}
	d_H(K,J) := \inf\{\varepsilon>0: K \subseteq J_\varepsilon\text{ and }J \subseteq K_\varepsilon\},
\end{equation*}
is a metric on $\comp(X)$.
\begin{lem}\label{app_gen_meas_selec_lem0}
	Let $f:X \to \mathbb{R}$ be upper semicontinuous, set $f_K := \sup_{x \in K}f(x)$ for $K \in \comp(X)$ and define $F: \comp(X)\to \comp(X)$ by
	$$F: K \mapsto \{x \in K: f(x) = f_K\}.$$
	The maps $K \mapsto f_K$ and $K \mapsto F(K)$ are Borel maps from $\comp(X)$ to $\mathbb{R}$ and $\comp(X)$, respectively.
\end{lem}

\begin{lem}\label{app_gen_meas_selec_lem1}
	Let $Y$ be a further metric space and $y \mapsto K_y$ a map from $Y$ to $\comp(X)$. Suppose for any $(y_n)_{n \in \mathbb{N}}$, $y \in Y$ with $y_n \longrightarrow y$ as $n \to \infty$ and for any $x_n \in K_{y_n}$, there exists a limit point $x$ of $(x_n)_{n \in \mathbb{N}}$ such that $x \in K_y$. Then, $y \mapsto K_y$ is Borel measurable from $Y$ to $(\comp(X),d_H)$.
\end{lem}
\begin{lem}\label{app_gen_meas_selec_lem2}
	Let $(E,\mathcal{F})$ be a measurable space and $q \mapsto K_q$ a measurable map from $E$ to $(\comp(X),d_H)$. Then, there is a $\mathcal{F}/\mathcal{B}(X)$-measurable map $h: E \to X$ such that $h(q) \in K_q$ for every $q \in E$.
\end{lem}

\bibliography{thesis}

\end{document}